\documentclass{article}

\usepackage{geometry, amsmath, amsthm, amssymb, bm, color, framed, graphicx, hyperref, mathrsfs,enumerate}

\usepackage{youngtab}
\usepackage[OT2,T1]{fontenc}
\DeclareSymbolFont{cyrletters}{OT2}{wncyr}{m}{n}
\DeclareMathSymbol{\Sha}{\mathalpha}{cyrletters}{"58}

\newtheorem{theorem}{Theorem}[section]
\newtheorem{lemma}[theorem]{Lemma}
\newtheorem{corollary}[theorem]{Corollary}
\newtheorem{definition}[theorem]{Definition}
\newtheorem{example}[theorem]{Example}
\newtheorem{remark}[theorem]{Remark}
\theoremstyle{definition}\newtheorem{prop}[subsection]{}

\title{Cyclotomy, cyclotomic cosets and arithmetic properties of some families in $\frac{\mathbb{F}_l[x]}{\langle x^{p^sq^t}-1\rangle}$}
\author{Juncheng Zhou$^{1}$ and Hongfeng Wu$^{1}$\footnote{Corresponding author.}
\setcounter{footnote}{-1}
\footnote{E-Mail addresses:
2024312050101@mail.ncut.edu.cn(J. Zhou), whfmath@gmail.com(H. Wu)}
\\
{1.~College of Science, North China University of technology, Beijing, China}}
\date{}

\begin{document}
\maketitle

\bibliographystyle{abbrv}
\thispagestyle{plain}
\setcounter{page}{1}

\begin{abstract}
Arithmetic properties of some families in $\frac{\mathbb{F}_l[x]}{\langle x^{p^sq^t}-1\rangle}$ are obtained by using the 
cyclotomic classes of order 2 with respect to $n=p^sq^t$, where $p\equiv3 \mathrm{mod} 4$, $\gcd(\phi(p^s),\phi(q^t))=2$, $l$ is 
a primitive root modulo $q^t$ and $\mathrm{ord}_{p^s}(l)=\phi(p^s)/2$. The form of these cyclotomic classes enables us to 
further generalize the results obtained in \cite{ref1}. The explicit expressions of primitive idempotents of minimal 
ideals in $\frac{\mathbb{F}_l[x]}{\langle x^{p^sq^t}-1\rangle}$ are also obtained.
\end{abstract}

\section{Introduction}

Theory of cyclotomy is a significant  tool in research of codes. Classical cyclotomy means cyclotomy with respect to an odd prime. C. Ding and T. Helleseth \cite{ref2} introduced a generalized cyclotomy of order 2 with respect to a composite number $n$ which includes classical cyclotomy as a special case. The last two decades have witness a lot of research on cyclic codes by using the theory of cyclotomy, for example, see \cite{ref1},\cite{ref2},\cite{ref4},\cite{ref6},\cite{ref8},\cite{ref11}$\sim$\cite{ref15}.

Let $\mathbb{F}_l$ be a finite field of order $l$, where $l$ is a prime. A cyclic code of length $n$ over $\mathbb{F}_l$ is viewed as an ideal in the ring $\frac{\mathbb{F}_l[x]}{\langle x^n-1\rangle}$, where $\gcd(n,l)=1$. It is well known that an ideal in $\frac{\mathbb{F}_l[x]}{\langle x^n-1\rangle}$ can be written as a direct sum of minimal ideals. Further, each minimal ideal (minimal cyclic code) is generated by a primitive idempotent. Thus it is useful to obtain the explicit expressions for primitive idempotents in $\frac{\mathbb{F}_l[x]}{\langle x^n-1\rangle}$. 

A lot of papers investigate cyclic codes and their primitive idempotents by using the theory of cyclotomy.
In \cite{ref2}, generalized cyclotomy of order 2 has been applied for constructing codes of length $n=p_1^{e_1}\cdots p_t^{e^t}$ with each $p_i$ being odd prime. 
In \cite{ref6}, C.Ding study the codes of length $pq$ over $\mathbb{F}_l$ by using generalized cyclotomy, where $l$ is a quadratic residue modulo both $p$ and $q$. 

In \cite{ref7}, R. Singh and M. Prunthi obtain the explicit expressions of primitive of irreducible quadratic residue cyclic codes of length $p^sq^t$ and $\mathrm{ord}_{p^s}(l)=\phi(p^s)/2$, $\mathrm{ord}_{q^t}(l)=\phi(q^t)/2$, which in fact generalizes C. Ding's results.
In\cite{ref13}, F. Li and Q. Yue investigate irreducible cyclic code of length $n=p^sq^t$ over $\mathbb{F}_l$ with $pq\mid l-1$.
In \cite{ref14}, Z. Shi and F. Fu found the primitive idempotents of irreducible constacyclic codes and LCD cyclic codes of length $n=p^sq^t$ over $\mathbb{F}_l$ with $\gcd(pq,l(l-1))=1$.
In \cite{ref4}, G.K. Backshi and M. Raka obtained minimal cyclic codes of length $p^sq$, where $l$ is a primitive root of both  $p^s$ and $q$, $\gcd(\phi(p)/2,\phi(q)/2)=1$.
In \cite{ref15}, A.Sahni and P.T. Sehgal also investigated minimal cyclic codes of length $p^sq$, where $l$ is a primitive root of both  $p^s$ and $q$, $\gcd(\phi(p),\phi(q))=d$.

S. Jain, S. Betra and K. Kmar consider a case different from all research above. In \cite{ref1}\cite{ref8}\cite{ref11}, $p,q$ are distinct odd primes, $p\equiv3 \mathrm{mod}  4$ and $\gcd(\phi(p^s),\phi(q^t))=2$. $l$, the order of a finite field, is a primitive root modulo $q$ and $\mathrm{ord}_{p^s}(l)=\phi(p^s)/2$. However, S. Jain et al. only investigate primitive idempotents when $n=pq$. Arithmetic properties of some families  and primitive idempotents in  $\frac{\mathbb{F}_l[x]}{\langle x^{p^sq^t}-1\rangle}$, where $l$ is a primitive root modulo $q^t$ and $\mathrm{ord}_{p^s}(l)=\phi(p^s)/2$, are unstudied.

In this paper, we first generalize the results of S. Jain et al. and obtained the arithmetic properties of some families in $\frac{\mathbb{F}_l[x]}{\langle x^{p^sq^t}-1\rangle}$, where $p,q$ are distinct odd primes, $p\equiv3 \mathrm{mod}  4$, $\gcd(\phi(p^s),\phi(q^t))=2$, $\mathrm{ord}_{p^s}(l)=\phi(p^s)/2$ and $\mathrm{ord}_{q^t}(l)=\phi(q^t)$. In Section \ref{sec2} we give the cyclotomic classes of Order 2 and write $l$-cyclotomic cosets modulo $p^sq^t$ in an additive form in order to facilitate the follow study. In Section \ref{sec3}, we calculate the cyclotomic number of order 2 and investigate relationships between cyclotomic cosets.

In Section \ref{sec4}, we investigate arithmetic properties of the polynomials in the ring $\frac{\mathbb{F}_l[x]}{\langle x^{p^sq^t}-1\rangle}$ that has the form\[\chi_\gamma=\sum_{i\in C_\gamma}x^i,\]where $C_\gamma$ is the cyclotomic coset containing $\gamma$. The results in Section \ref{sec3} will be used to prove the theorems in Section \ref{sec4}. For convenience, lemmas in Section \ref{sec3} and theorems in Section \ref{sec4} are one-to-one, that is, Lemma 3.n $\mapsto$ Theorem 4.n-1, for $2\leqslant n\leqslant 17$.

In Section \ref{sec5} we devoted to calculate the explicit expressions of primitive idempotents. Theorem 6, Chapter 8 of \cite{ref5} gives the expressions over $\mathbb{F}_2$, and G.K. Bassi and M. Raka \cite{ref4} generalize it into $\mathbb{F}_l$ for any primes. By classifying $p^sq^t$th roots of unit, we can see that primitive idempotents can be written in the form of linear combination of $\chi_\gamma$, which is defined in Section \ref{sec4}. We then obtain the explicit expressions of primitive idempotents of minimal ideals in $\frac{\mathbb{F}_l[x]}{\langle x^{p^sq^t}-1\rangle}$.

The terminology and assumptions throughout this paper are as follows:

\begin{enumerate}[(1)]
\item $p,q,l$ are three distinct primes with $p$ and $q$ are odd and $p\equiv3 \mathrm{mod}  4$.
\item $\mathbb{Z}_m^*$ denotes the set of units modulo m, for any positive integer m.
\item $\mathrm{ord}_m(k)$ denotes the multiplicative order of $k$ modulo $m$ and $\phi(\cdot)$ denotes Euler's phi function.
\item $\mathrm{ord}_{p^s}(l)=\phi(p^s)/2$, $\mathrm{ord}_{q^t}(l)=\phi(q^t)$ and $\gcd(\phi(p^s),\phi(q^t))=2$.
\item $R_k$ denotes the set of quadratic residues in $\mathbb{Z}_{p^k}^*$ and $N_k$ denotes the set of quadratic nonresidual in $\mathbb{Z}_{p^k}^*$. It is easy to find that\[R_k=\{x+p\lambda:x\in R_1,0\leqslant \lambda\leqslant p^{k-1}-1\}\] and \[N_k=\{x+p\lambda:x\in N_1,0\leqslant \lambda\leqslant p^{k-1}-1\}\]
\end{enumerate}

\section{Cyclotomic classes of order 2 and $l$-cyclotomic cosets modulo $p^sq^t$}
\label{sec2}
\subsection{$l$-cyclotomic cosets}

In this section, we obtain all the $l$-cyclotomic cosets modulo $n=p^sq^t$, when $\mathrm{ord}_{p^s}(l)=\phi(p^s)$/2, $\mathrm{ord}_{q^t}(l)=\phi(q^t)$ and $\gcd(\phi(p^s),\phi(q^t))=2$. To obtain these cosets, we first prove some results.

\begin{definition}
The $l$-cyclotomic coset modulo $n$ containing $\gamma$ denoted by $C_\gamma$ is defined as\[C_\gamma=\{\gamma l^i \mathrm{mod}  n:i=0,1,\cdots, \tau-1\},\]where $\tau$ is the least positive integer such that \[\gamma l^\tau\equiv\gamma \mathrm{mod}  n.\]
\end{definition}
We know that these cyclotomic coset partition the set $\{0,1,\cdots,n-1\}$. A subset $\{\gamma_1,\gamma_2,\cdots,\gamma_m\}$ of $\mathbb{Z}_n$ is said to be complete set of representative of $l$-cyclotomic cosets modulo $n$ if $C_{\gamma_1},C_{\gamma_2},\cdots,C_{\gamma_m}$ are distinct and\[\bigcup_{i=1}^mC_{\gamma_i}=\mathbb{Z}_n\]

\begin{lemma}[\cite{ref1}, Lemma 2.1, p.3]
$\mathrm{ord}_n(l)=\phi(n)/2$.
\end{lemma}

\begin{lemma}[\cite{ref1}, Lemma 2.2, p.3]
If $g$ is primitive root modulo $q^t$, then $g$ is also a primitive root modulo $q^{t-j}$ with $0\leqslant j<t$.
\end{lemma}

\begin{lemma}[\cite{ref1}, Lemma 2.3, p.3]
If $\mathrm{ord}_{p^s}(l)=\phi(p^s)/2$, then $\mathrm{ord}_{p^{s-i}}(l)=\phi(p^{s-i})/2$ with $0\leqslant i<s$.
\end{lemma}

Now we will find all the $l$-cyclotomic cosets modulo $p^sq^t$ using the results prove above.
\begin{theorem}
There are $(2s+1)(t+1)$ $l$-cyclotomic cosets modulo $p^sq^t$ which are given as follows:
\[
\begin{aligned}
&C_0=\{0\};\\
&C_{p^sq^j}=\{p^sq^jl^m:0\leqslant m\leqslant\phi(q^{t-j})-1\},\quad0\leqslant j\leqslant t-1;\\
&C_{p^iq^j}=\{p^iq^jl^m:0\leqslant m\leqslant\phi(p^{s-i}q^{t-j})/2-1\},\quad0\leqslant i\leqslant s-1,0\leqslant j\leqslant t;\\
&C_{p^iq^jg}=\{p^iq^jgl^m:0\leqslant m\leqslant\phi(p^{s-i}q^{t-j})/2-1\},\quad0\leqslant i\leqslant s-1,0\leqslant j\leqslant t.
\end{aligned}
\]\label{thm2.1}
\end{theorem} 

\begin{proof}
Obviously, the sets given above are distinct. Also,\[\# C_0+\sum_{j=0}^{t-1}\#C_{p^sq^j}+\sum_{i=0}^{s-1}\sum_{j=0}^t\#C_{p^iq^j}+\sum_{i=0}^{s-1}\sum_{j=0}^t\#C_{p^iq^jg}=n,\]which implies that these are the all $l$-cyclotomic cosets modulo $n$, and hence this proves the results.
\end{proof}

\subsection{Cyclotomic classes of order 2}

\begin{lemma}[\cite{ref9}, Theorem 8.10, p.162]
An integer $n>1$ has a primitive root if and only if $n=2,4,p^s$ or $2p^s$ for some odd prime $p$.
\end{lemma}

Let $g_1$ and $g_2$ be the primitive roots modulo $p^s$ and $q^t$, respectively, and $g$ be the solution of simultaneous congruences\[
\begin{aligned}
g&\equiv g_1 \mathrm{mod}  p^s,\\
g&\equiv g_2 \mathrm{mod}  q^t.
\end{aligned}
\]
Then the existence of $g$ is guaranteed by Chinese Remainder Theorem. $g$ is a common primitive root of both $p^s$ and $q^t$. 

Further, let $v$ be the solution of simultaneous congruences\[\begin{aligned}
v&\equiv 1 \mathrm{mod}  p^s,\\
v&\equiv g \mathrm{mod}  q^t.\end{aligned}
\]
Then the cyclotomic classes of order 2 with respect to $p^sq^t$ are given by\[
D_0^{p^sq^t}=\left\{g^{2i}:i=0,1,\cdots,\frac{\phi(p^sq^t)}{4}-1\right\}\cup\left\{g^{2i}v:i=0,1,\cdots,\frac{\phi(p^sq^t)}{4}-1\right\},
\]
and
\[D_1^{p^sq^t}=gD_0^{p^sq^t}\]

\begin{lemma}[\cite{ref10}, Theorem 83]
An integer a is quadratic residue modulo an odd prime $p$ if and only if $a^{\phi(p)/2}\equiv1 \mathrm{mod}  p$.\label{lem2.7}
\end{lemma}

\begin{lemma}
Let $p$ be an odd prime and $a$ be an integer coprime with $p$, then $a$ is an quadratic residue modula $p^s$ if and only if $a$ is a quadratic residue modulo $p$.
\end{lemma}

\begin{proof}
Obviously, $a$ is an quadratic residue modula $p$ if $a$ is a quadratic residue modulo $p^s$. If $a$ is an quadratic residue modula $p^i$, there exists an integer $b$ such that $b^2\equiv a \mathrm{mod}  p^i$. We claim that there exists an integer $k\in\{0,1,\cdots,p-1\}$ such that $(b+kp^i)^2\equiv a \mathrm{mod}  p^{i+1}$ for $(2,p)=1$, and hence this proves the result.
\end{proof}

\begin{lemma}
Let $p$ and $q$ are two distinct odd primes. If $x$ runs through reduced residue system modulo $q^t$ and $y$ runs through reduced system modulo $p^s$, then the set $\{p^sx+q^ty\}$ containing $\phi(p^s)\phi(q^t)$ number of integers, forms the reduced residue system modulo $p^sq^t$.\label{lem2.9}
\end{lemma}

\begin{proof}
Obviously, elements in the set $\{p^sx+q^ty\}$ are coprime with both $p$ and $q$. If $p^sx+q^ty=p^sx'+q^ty'$, then $p^s(x-x')=-q^t(y-y')$. For $p$, $q$ are distinct, $x\equiv x' \mathrm{mod}  q^t$ and $y\equiv y' \mathrm{mod}  p^s$. Also, Since there are exactly $\phi(p^s)\phi(q^t)$ elements in the reduced system modulo $p^sq^t$, the result is proved.
\end{proof}

\begin{theorem}
Let $v\equiv1 \mathrm{mod}  p^s$ and $v\equiv g \mathrm{mod}  q^t$, then
\begin{enumerate}[(i)]
\item If $q$ is a quadratic residue modulo $p$, then \[D_0^{p^sq^t}=\left\{q^tx+p^sy:x\in R_s, y\in\mathbb{Z}_{q^t}^*\right\}\]
and
\[D_1^{p^sq^t}=\left\{q^tx+p^sy:x\in N_s, y\in\mathbb{Z}_{q^t}^*\right\}.\]

\item If $q$ is a quadratic nonresidue modulo $p$, then \[
\begin{aligned}
&D_0^{p^sq^t}=\left\{q^tx+p^sy:x\in N_s, y\in\mathbb{Z}_{q^t}^*\right\},\\
&D_1^{p^sq^t}=\left\{q^tx+p^sy:x\in R_s, y\in\mathbb{Z}_{q^t}^*\right\},
\end{aligned}\]
if $2\nmid t$ and
\[
\begin{aligned}
&D_0^{p^sq^t}=\left\{q^tx+p^sy:x\in R_s, y\in\mathbb{Z}_{q^t}^*\right\},\\
&D_1^{p^sq^t}=\left\{q^tx+p^sy:x\in N_s, y\in\mathbb{Z}_{q^t}^*\right\},
\end{aligned}.\]
if $2\mid t$.
\end{enumerate}
\end{theorem}

\begin{proof}
We only prove the first result, the second result can be proved similarly. By the definition of $D_0^{p^sq^t}$, elements in $D_0^{p^sq^t}$ are all quadratic residue modulo $p^s$ and $D_0^{p^sq^t}( \mathrm{mod}  p^s)$ contains $\phi(q^t)$ copies of $R_s$. Similarly, $D_1^{p^sq^t}( \mathrm{mod}  p^s)$ contains $\phi(q^t)$ copies of $N_s$.

Let $\omega_r=\{q^tx+p^sy:x\in R_s, y\in\mathbb{Z}_{q^t}^*\}$ and $\omega_n=\{q^tx+p^sy:x\in N_s, y\in\mathbb{Z}_{q^t}^*\}$, then by Lemma \ref{lem2.9}, $\omega_r\sqcup \omega_g=\mathbb{Z}_{p^sq^t}^*$. If $q$ is a quadratic residue modulo $p$, then $\omega_r( \mathrm{mod}  p^s)$ contains $\phi(q^t)$ copies of $R_s$ and $\omega_n( \mathrm{mod}  p^s)$ contains $\phi(q^t)$ copies of $N_s$. So $\omega_r=D_0^{p^sq^t}$ and $\omega_n=D_1^{p^sq^t}$.
\end{proof}

\begin{lemma}
$C_1=D_0^{p^sq^t}$ and $C_g=D_1^{p^sq^t}$.\label{lem2.8}
\end{lemma}

\begin{proof}
By Lemma \ref{lem2.7} $l$ is a quadratic residue modulo $p$. Since $\mathrm{ord}_{p^sq^t}=\phi(p^sq^t)/2$, we can see that $C_1( \mathrm{mod}  p^s)$ contains $\phi(q^t)$ copies of $R_s$ and $C_g( \mathrm{mod}  p^s)$ contains $\phi(q^t)$ copies of $N_s$. Also, $C_1\sqcup C_g=\mathbb{Z}_{p^sq^t}^*$, so we have $C_1=D_0^{p^sq^t}$ and $C_g=D_1^{p^sq^t}$.
\end{proof}

\begin{theorem}
The $l$-cyclotomic coset obtained in Theorem \ref{thm2.1} can also be represented as follows:

\begin{enumerate}[(i)]
\item If $q$ is a quadratic residue modulo $p$, then\[
\begin{aligned}
&C_0=\{0\};\\
&C_{p^iq^j}=\left\{p^iq^tx+p^sq^jy:x\in R_{s-i},q\in \mathbb{Z}_{q^{t-j}}^*\right\},\quad 0\leqslant i\leqslant s-1, 0\leqslant j\leqslant t-1;\\
&C_{p^iq^t}=\left\{p^iq^tx:x\in R_{s-i}\right\},\quad0\leqslant i\leqslant s-1;\\
&C_{p^iq^jg}=\left\{p^iq^tx+p^sq^jy:x\in N_{s-i},q\in \mathbb{Z}_{q^{t-j}}^*\right\},\quad 0\leqslant i\leqslant s-1, 0\leqslant j\leqslant t-1;\\
&C_{p^iq^tg}=\left\{p^iq^tx:x\in N_{s-i}\right\},\quad0\leqslant i\leqslant s-1;\\
&C_{p^sq^j}=\left\{p^sq^jy:y\in\mathbb{Z}_{q^{t-j}}^*\right\},\quad 0\leqslant j\leqslant t-1.
\end{aligned}\]
\item If $q$ is a quadratic nonresidue modulo $p$, then\[
\begin{aligned}
&C_0=\{0\};\\
&C_{p^iq^j}=\left\{p^iq^tx+p^sq^jy:x\in N_{s-i},q\in \mathbb{Z}_{q^{t-j}}^*\right\},\quad 0\leqslant i\leqslant s-1, 0\leqslant j\leqslant t-1,2\nmid t-j;\\
&C_{p^iq^j}=\left\{p^iq^tx+p^sq^jy:x\in R_{s-i},q\in \mathbb{Z}_{q^{t-j}}^*\right\},\quad 0\leqslant i\leqslant s-1, 0\leqslant j\leqslant t-1,2\mid t-j;\\
&C_{p^iq^t}=\left\{p^iq^tx:x\in R_{s-i}\right\},\quad0\leqslant i\leqslant s-1;\\
&C_{p^iq^jg}=\left\{p^iq^tx+p^sq^jy:x\in R_{s-i},q\in \mathbb{Z}_{q^{t-j}}^*\right\}, \quad0\leqslant i\leqslant s-1, 0\leqslant j\leqslant t-1,2\nmid t-j;\\
&C_{p^iq^jg}=\left\{p^iq^tx+p^sq^jy:x\in N_{s-i},q\in \mathbb{Z}_{q^{t-j}}^*\right\},\quad0\leqslant i\leqslant s-1, 0\leqslant j\leqslant t-1, 2\mid t-j;\\
&C_{p^iq^tg}=\left\{p^iq^tx:x\in N_{s-i}\right\},\quad0\leqslant i\leqslant s-1;\\
&C_{p^sq^j}=\left\{p^sq^jy:y\in\mathbb{Z}_{q^{t-j}}^*\right\},\quad 0\leqslant j\leqslant t-1..
\end{aligned}\]
\end{enumerate}\label{2.12}
\end{theorem}

\begin{proof}
Note that $C_{p^iq^j}=p^iq^jD_0^{p^{s-i}q^{t-j}}$ and $C_{p^iq^jg}=p^iq^jD_1^{p^{s-i}q^{t-j}}$, then it can be proved easily using Lemma \ref{lem2.8}
\end{proof}

\section{Cyclotomic number of order 2}
\label{sec3}
\begin{lemma}
Let $p$ be a prime of the form $4k-1$, then

\begin{enumerate}[(i)]
\item For any $r\in R_s$, we have\[\#(r+R_s)\cap R_s=\#(r+N_s)\cap R_s=\#(r+N_s)\cap N_s=\frac{p-3}{4}p^{s-1},\]and\[\#(r+R_s)\cap N_s=\frac{p+1}{4}p^{s-1}.\]
\item For any $n\in N_s$, we have\[\#(n+N_s)\cap N_s=\#(n+R_s)\cap R_s=\#(n+R_s)\cap N_s=\frac{p-3}{4}p^{s-1},\]and\[\#(n+N_s)\cap R_s=\frac{p+1}{4}p^{s-1}.\]
\end{enumerate}\label{3.1}
\end{lemma}

\begin{proof}
We only prove the first result, the second one can be proved similarly. By Lemma \ref{lem2.7}, -1 is a quadratic nonresidue modulo $p$. so there is no element in $(r+R_s)\cap p\mathbb{Z}_{p^s}$, then\[\#(r+R_s)\cap R_s+(r+R_s)\cap N_s=\#(r+R_s)\cap\mathbb{Z}_{p^s}-\#(r+R_s)\cap p\mathbb{Z}_{p^s}=\frac{p-1}{2}p^{s-1}.\]
Since $-r$ is nonresidue modulo $p$, there are $p^{s-1}$ elements in $(r+N_s)\cap p\mathbb{Z}_{p^s}$, then
\[\#(r+N_s)\cap R_s+(r+N_s)\cap N_s=\#(r+N_s)\cap\mathbb{Z}_{p^s}-\#(r+N_s)\cap p\mathbb{Z}_{p^s}=\frac{p-3}{2}p^{s-1}.\]
Also, since $r+p\mathbb{Z}_{p^s}\subset R_s$,\[\#(r+R_s)\cap R_s+\#(r+N_s)\cap R_s=\# R_s-\#(r+p\mathbb{Z}_{p^s})\cap R_{s}=\frac{p-3}{2}p^{s-1}\]and\[\#(r+R_s)\cap N_s+\#(r+N_s)\cap N_s=\# N_s-\#(r+p\mathbb{Z}_{p^s})\cap N_{s}=\frac{p-1}{2}p^{s-1}.\]
By the four equations above, we have
\[\#(r+R_s)\cap R_s=\#(r+N_s)\cap R_s=\#(r+N_s)\cap N_s=\frac{p-3}{4}p^{s-1},\]and\[\#(r+R_s)\cap N_s=\frac{p+1}{4}p^{s-1}.\]
\end{proof}

For simplicity, we denote $C_{ij}=C_{p^iq^j}$ and $C_{ij}^*=C_{p^iq^jg}$. In particular, $C_{ij}=C_{ij}^*$ if $i=s$ and $C_{st}=C_0$. Then by Theorem \ref{2.12} and Lemma \ref{3.1} , we can easily obtain the results in Lemma \ref{3.2} $\sim$ Lemma \ref{3.17}.

\begin{lemma}
Let $a\in C_{ij}$ with $0\leqslant i\leqslant s-1$ and $0\leqslant j\leqslant t-1$, then
\begin{enumerate}[(i)]
\item $\#(a+C_{ij})\cap C_{ij}=\frac{p-3}{4}p^{s-i-1}(q-2)q^{t-j-1}$;
\item $\#(a+C_{ij})\cap C_{ij}^*=\frac{p+1}{4}p^{s-i-1}(q-2)q^{t-j-1}$;
\item If $q$ is a quadratic residue modulo $p$, then for $j<m\leqslant t$,\[\#(a+C_{ij})\cap C_{im}=\frac{p-3}{4}p^{s-i-1}\phi(q^{t-m})\]
and\[\#(a+C_{ij})\cap C_{im}^*=\frac{p+1}{4}p^{s-i-1}\phi(q^{t-m}).\]
\item If $q$ is a quadratic nonresidue modulo $p$, then for $j<m\leqslant t$,\[\#(a+C_{ij})\cap C_{im}=\biggl\{\begin{array}{l}\frac{p+1}{4}p^{s-i-1}\phi(q^{t-m}),\quad 2\nmid m-j,\\\frac{p-3}{4}p^{s-i-1}\phi(q^{t-m}),\quad 2\mid m-j,\end{array}\]
and
\[\#(a+C_{ij})\cap C_{im}^*=\biggl\{\begin{array}{l}\frac{p-3}{4}p^{s-i-1}\phi(q^{t-k}),\quad 2\nmid m-j,\\\frac{p+1}{4}p^{s-i-1}\phi(q^{t-m}),\quad 2\mid m-j.\end{array}\]
\end{enumerate}\label{3.2}
\end{lemma}

\begin{lemma}
Let $a^*\in C_{ij}^*$ with $0\leqslant i\leqslant s-1$ and $0\leqslant j\leqslant t-1$, then
\begin{enumerate}[(i)]
\item $\#(a^*+C_{ij}^*)\cap C_{ij}=\frac{p+1}{4}p^{s-i-1}(q-2)q^{t-j-1}$;
\item $\#(a^*+C_{ij}^*)\cap C_{ij}^*=\frac{p-3}{4}p^{s-i-1}(q-2)q^{t-j-1}$;
\item If $q$ is a quadratic residue modulo $p$, then for $j<m\leqslant t$,\[\#(a^*+C_{ij}^*)\cap C_{im}=\frac{p+1}{4}p^{s-i-1}\phi(q^{t-m})\]
and\[\#(a^*+C_{ij}^*)\cap C_{im}^*=\frac{p-3}{4}p^{s-i-1}\phi(q^{t-m}).\]

\item If $q$ is a quadratic nonresidue modulo $p$, then for $j<m\leqslant t$,\[\#(a^*+C_{ij}^*)\cap C_{im}=\biggl\{\begin{array}{l}\frac{p-3}{4}p^{s-i-1}\phi(q^{t-m}),\quad 2\nmid m-j,\\\frac{p+1}{4}p^{s-i-1}\phi(q^{t-m}),\quad 2\mid m-j,\end{array}\]
and
\[\#(a^*+C_{ij}^*)\cap C_{im}^*=\biggl\{\begin{array}{l}\frac{p+1}{4}p^{s-i-1}\phi(q^{t-k}),\quad 2\nmid m-j,\\\frac{p-3}{4}p^{s-i-1}\phi(q^{t-m}),\quad 2\mid m-j.\end{array}\]
\end{enumerate}
\end{lemma}

\begin{lemma}
Let $a\in C_{ij}$, $0\leqslant i\leqslant s-1$ and $0\leqslant j\leqslant t-1$, then
\begin{enumerate}[(i)]
\item $\#(a+C_{ij}^*)\cap C_{ij}=\#(a+C_{ij}^*)\cap C_{ij}^*=\frac{p-3}{4}p^{s-i-1}(q-2)q^{t-j-1}$.

\item For $j<m\leqslant t$, we have\[\#(a+C_{ij}^*)\cap C_{im}=\#(a+C_{ij}^*)\cap C_{im}^*=\frac{p-3}{4}p^{s-i-1}\phi(q^{t-m}).\]

\item For $i<k\leqslant s-1$ and $j<m\leqslant t$, we have\[\#(a+C_{ij}^*)\cap C_{km}=\#(a+C_{ij}^*)\cap C_{km}^*=\phi(p^{s-k})\phi(q^{t-m})/2.\]
\item $\#(a+C_{ij}^*)\cap C_{sj}=(q-2)q^{t-j-1}$.
\item For $j<m\leqslant t$, we have\[\#(a+C_{ij}^*)\cap C_{sm}=\phi(q^{t-m}).\]
\item For $i<k\leqslant s-1$, we have
\[\#(a+C_{ij}^*)\cap C_{kj}=\#(a+C_{ij}^*)\cap C_{kj}^*=\phi(p^{s-k})(q-2)q^{t-j-1}/2.\]
\end{enumerate}
\end{lemma}

\begin{lemma}
Let $a\in C_{sj}$ and $0\leqslant j\leqslant t-1$, then
\begin{enumerate}[(i)]
\item $\#(a+C_{sj})\cap C_{sj}=(q-2)q^{t-j-1}$.
\item For $j<m\leqslant t$, we have\[\#(a+C_{sj})\cap C_{sm}=\phi(q^{t-m}).\]
\end{enumerate}
\end{lemma}

\begin{lemma}
Let $a\in C_{it}$, $a^*\in C_{it}^*$ and $0\leqslant i\leqslant s-1$, then
\begin{enumerate}[(i)]
\item $\#(a+C_{it})\cap C_{it}=\#(a^*+C_{it}^*)\cap C_{it}^*=\frac{p-3}{4}p^{s-i-1}$.
\item $\#(a+C_{it})\cap C_{it}^*=\#(a^*+C_{it}^*)\cap C_{it}=\frac{p+1}{4}p^{s-i-1}$.
\item $\#(a+C_{it}^*)\cap C_{it}=\#(a+C_{it}^*)\cap C_{it}^*=\frac{p-3}{4}p^{s-i-1}$.
\item For $i<k\leqslant s-1$, we have\[\#(a+C_{it}^*)\cap C_{kt}=(a+C_{it}^*)\cap C_{kt}^*=\frac{\phi(p^{s-k})}{2}.\]
\item $\#(a+C_{it}^*)\cap C_0=1$.
\end{enumerate}
\end{lemma}

\begin{lemma}
Let $a\in C_{ij}$, $a^*\in C_{ij}$, $0\leqslant i<i'\leqslant s-1$ and $0\leqslant j<j'\leqslant t$, then
\begin{enumerate}[(i)]
\item $\#(a+C_{i'j'})\cap C_{ij}=\#(a+C_{i'j'}^*)\cap C_{ij}=\frac{\phi(p^{s-i'})}{2}\phi(q^{t-j'})$.
\item $\#(a^*+C_{i'j'})\cap C_{ij}^*=\#(a^*+C_{i'j'}^*)\cap C_{ij}^*=\frac{\phi(p^{s-i'})}{2}\phi(q^{t-j'})$.
\end{enumerate} 
\end{lemma}

\begin{lemma}
Let $a\in C_{ij}$, $a^*\in C_{ij}$, $0\leqslant i<i'\leqslant s-1$ and $0\leqslant j'<j\leqslant t$, then
\begin{enumerate}[(i)]
\item $\#(a+C_{i'j'})\cap C_{ij'}=\#(a+C_{i'j'}^*)\cap C_{ij'}=\frac{\phi(p^{s-i'})}{2}\phi(q^{t-j})$.
\item $\#(a^*+C_{i'j'})\cap C_{ij'}^*)=\#(a^*+C_{i'j'}^*)\cap C_{ij'}^*)=\frac{\phi(p^{s-i'})}{2}\phi(q^{t-j})$.
\end{enumerate}
\end{lemma}

\begin{lemma}
Let $a\in C_{ij}$, $0\leqslant i\leqslant s-1$ and $0\leqslant j<j'\leqslant t$, then
\begin{enumerate}[(i)]
\item $\#(a+C_{ij'})\cap C_{ij}=\frac{p-3}{4}p^{s-i-1}\phi(q^{t-j'})$.
\item Suppose $q$ is a quadratic residue modulo $p$, then\[\#(a+C_{ij'})\cap C_{ij}^*=\frac{p+1}{4}p^{s-i-1}\phi(q^{t-j'}).\]
\item If $q$ is a quadratic nonresidue modulo $p$ and $2\mid j'-j$, then\[\#(a+C_{ij'})\cap C_{ij}^*=\frac{p+1}{4}p^{s-i-1}\phi(q^{t-j'}).\]
\item If $q$ is a quadratic nonresidue modulo $p$ and $2\nmid j'-j$, then
\[\#(a+C_{ij'})\cap C_{ij}^*=\frac{p-3}{4}p^{s-i-1}\phi(q^{t-j'}).\]
Also, for $i<k\leqslant s-1$,
\[\#(a+C_{ij'})\cap C_{kj}=\#(a+C_{ij'})\cap C_{kj}^*=\frac{\phi(p^{s-k})}{2}\phi(q^{t-j'}),\]
and
\[\#(a+C_{ij'})\cap C_{sj}=\phi(q^{t-j'}).\]
\end{enumerate}
\end{lemma}

\begin{lemma}
Let $a^*\in C_{ij}^*$, $0\leqslant i\leqslant s-1$ and $0\leqslant j<j'\leqslant t$, then
\begin{enumerate}[(i)]
\item $\#(a^*+C_{ij'}^*)\cap C_{ij}^*=\frac{p-3}{4}p^{s-i-1}\phi(q^{t-j'})$.
\item Suppose $q$ is a quadratic residue modulo $p$, then\[\#(a^*+C_{ij'}^*)\cap C_{ij}=\frac{p+1}{4}p^{s-i-1}\phi(q^{t-j'}).\]
\item If $q$ is a quadratic nonresidue modulo $p$ and $2\mid j'-j$, then\[\#(a^*+C_{ij'}^*)\cap C_{ij}=\frac{p+1}{4}p^{s-i-1}\phi(q^{t-j'}).\]
\item If $q$ is a quadratic nonresidue modulo $p$ and $2\nmid j'-j$, then
\[\#(a^*+C_{ij'}^*)\cap C_{ij}=\frac{p-3}{4}p^{s-i-1}\phi(q^{t-j'}).\]
Also, for $i<k\leqslant s-1$,
\[\#(a^*+C_{ij'}^*)\cap C_{kj}=\#(a+C_{ij'})\cap C_{kj}^*=\frac{\phi(p^{s-k})}{2}\phi(q^{t-j'}),\]
and
\[\#(a^*+C_{ij'}^*)\cap C_{sj}=\phi(q^{t-j'}).\]
\end{enumerate}
\end{lemma}

\begin{lemma}
Let $a\in C_{ij}$, $0\leqslant i\leqslant s-1$ and $0\leqslant j<j'\leqslant t$, then
\begin{enumerate}[(i)]
\item If $q$ is a quadratic residue modulo $p$, then
\[
\#(a+C_{ij'}^*)\cap C_{ij}=(a+C_{ij'}^*)\cap C_{ij}^*=\frac{p-3}{4}p^{s-i-1}\phi(q^{t-j'}).
\]
Also, for $i<k\leqslant s-1$,
\[
\#(a+C_{ij'}^*)\cap C_{kj}=\#(a+C_{ij'}^*)\cap C_{kj}^*=\frac{\phi(p^{s-k})}{2}\phi(q^{t-j'}),
\]
and\[\#(a+C_{ij'}^*)\cap C_{sj}=\phi(q^{t-j'}).\]

\item If $q$ is a quadratic nonresidue modulo $p$ and $2\nmid j'-j$, then
\[
\#(a+C_{ij'}^*)\cap C_{ij}=\frac{p-3}{4}p^{s-i-1}\phi(q^{t-j'})
\]
and
\[\#(a+C_{ij'}^*)\cap C_{ij}^*=\frac{p+1}{4}p^{s-i-1}\phi(q^{t-j'}).\]

\item If $q$ is a quadratic nonresidue modulo $p$ and $2\mid j'-j$, then
\[
\#(a+C_{ij'}^*)\cap C_{ij}=(a+C_{ij'}^*)\cap C_{ij}^*=\frac{p-3}{4}p^{s-i-1}\phi(q^{t-j'}).
\]
Also, for $i<k\leqslant s-1$,
\[
\#(a+C_{ij'}^*)\cap C_{kj}=\#(a+C_{ij'}^*)\cap C_{kj}^*=\frac{\phi(p^{s-k})}{2}\phi(q^{t-j'}),
\]
and\[\#(a+C_{ij'}^*)\cap C_{sj}=\phi(q^{t-j'}).\]
\end{enumerate}
\end{lemma}

\begin{lemma}
Let $a^*\in C_{ij}^*$, $0\leqslant i\leqslant s-1$ and $0\leqslant j<j'\leqslant t$, then
\begin{enumerate}[(i)]
\item If $q$ is a quadratic residue modulo $p$, then
\[
\#(a^*+C_{ij'})\cap C_{ij}=(a^*+C_{ij'})\cap C_{ij}^*=\frac{p-3}{4}p^{s-i-1}\phi(q^{t-j'}).
\]
Also, for $i<k\leqslant s-1$,
\[
\#(a^*+C_{ij'})\cap C_{kj}=\#(a^*+C_{ij'})\cap C_{kj}^*=\frac{\phi(p^{s-k})}{2}\phi(q^{t-j'}),
\]
and\[\#(a^*+C_{ij'})\cap C_{sj}=\phi(q^{t-j'}).\]

\item If $q$ is a quadratic nonresidue modulo $p$ and $2\nmid j'-j$, then
\[
\#(a^*+C_{ij'})\cap C_{ij}=\frac{p+1}{4}p^{s-i-1}\phi(q^{t-j'})
\]
and
\[\#(a^*+C_{ij'})\cap C_{ij}^*=\frac{p-3}{4}p^{s-i-1}\phi(q^{t-j'}).\]

\item If $q$ is a quadratic nonresidue modulo $p$ and $2\mid j'-j$, then
\[
\#(a^*+C_{ij'})\cap C_{ij}=(a^*+C_{ij'})\cap C_{ij}^*=\frac{p-3}{4}p^{s-i-1}\phi(q^{t-j'}).
\]
Also, for $i<k\leqslant s-1$,
\[
\#(a^*+C_{ij'})\cap C_{kj}=\#(a^*+C_{ij'})\cap C_{kj}^*=\frac{\phi(p^{s-k})}{2}\phi(q^{t-j'}),
\]
and\[\#(a^*+C_{ij'})\cap C_{sj}=\phi(q^{t-j'}).\]
\end{enumerate}
\end{lemma}

\begin{lemma}
\begin{enumerate}[(i)]
\item Let $a\in C_{sj}$ and $0\leqslant j<j'\leqslant t$, then$\#(a+C_{sj'})\cap C_{sj}=\phi(q^{t-j'})$.
\item Let $a\in C_{ij}$, $0\leqslant i\leqslant s-1$ and $0\leqslant j< j'\leqslant t$, then
$\#(a+C_{sj'})\cap C_{ij}=\phi(q^{t-j'})$.
\item Let $a^*\in C_{ij}^*$, $0\leqslant i\leqslant s-1$ and $0\leqslant j< j'\leqslant t$, then
$\#(a^*+C_{sj'})\cap C_{ij}^*=\phi(q^{t-j'})$.
\item Let $a\in C_{ij}$, $0\leqslant i\leqslant s-1$ and $0\leqslant j'< j\leqslant t$, then
$\#(a+C_{sj'})\cap C_{ij'}=\phi(q^{t-j})$.
\item Let $a^*\in C_{ij}^*$, $0\leqslant i\leqslant s-1$ and $0\leqslant j'< j\leqslant t$, then
$\#(a^*+C_{sj'})\cap C_{ij'}^*=\phi(q^{t-j})$.
\end{enumerate}
\end{lemma}

\begin{lemma}
Let $a\in C_{ij}$, $0\leqslant i<i'\leqslant s-1$ and $0\leqslant j\leqslant t-1$, then
\begin{enumerate}[(i)]
\item $\#(a+C_{i'j})\cap C_{ij}=\frac{\phi(p^{s-i'})}{2}(q-2)q^{t-j-1}$.
\item If $q$ is a quadratic residue modulo $p$, then for $j<m\leqslant t$, we have
\[\#(a+C_{i'j})\cap C_{im}=\frac{\phi(p^{s-i'})}{2}\phi(q^{t-m}).\]
\item If $q$ is a quadratic nonresidue modulo $p$, then for $j<m\leqslant t$ and $2\nmid m-j$, we have\[\#(a+C_{i'j})\cap C_{im}^*=\frac{\phi(p^{s-i'})}{2}\phi(q^{t-m}).\]For $j<m\leqslant t$ and $2\mid m-j$, we have\[\#(a+C_{i'j})\cap C_{im}=\frac{\phi(p^{s-i'})}{2}\phi(q^{t-m}).\]
\end{enumerate}
\end{lemma}

\begin{lemma}
Let $a^*\in C_{ij}^*$, $0\leqslant i<i'\leqslant s-1$ and $0\leqslant j\leqslant t-1$, then
\begin{enumerate}[(i)]
\item $\#(a^*+C_{i'j}^*)\cap C_{ij}^*=\frac{\phi(p^{s-i'})}{2}(q-2)q^{t-j-1}$.
\item If $q$ is a quadratic residue modulo $p$, then for $j<m\leqslant t$, we have
\[\#(a^*+C_{i'j}^*)\cap C_{im}^*=\frac{\phi(p^{s-i'})}{2}\phi(q^{t-m}).\]
\item If $q$ is a quadratic nonresidue modulo $p$, then for $j<m\leqslant t$ and $2\nmid m-j$, we have\[\#(a^*+C_{i'j}^*)\cap C_{im}=\frac{\phi(p^{s-i'})}{2}\phi(q^{t-m}).\]For $j<m\leqslant t$ and $2\mid m-j$, we have\[\#(a^*+C_{i'j}^*)\cap C_{im}^*=\frac{\phi(p^{s-i'})}{2}\phi(q^{t-m}).\]
\end{enumerate}
\end{lemma}

\begin{lemma}
Let $a\in C_{ij}$, $0\leqslant i<i'\leqslant s-1$ and $0\leqslant j\leqslant t-1$, then
\begin{enumerate}
\item $\#(a+C_{i'j}^*)\cap C_{ij}=\frac{\phi(p^{s-i'})}{2}(q-2)q^{t-j-1}$.
\item If $q$ is a quadratic residue modulo $p$, then for $j<m\leqslant t$, we have
\[\#(a+C_{i'j}^*)\cap C_{im}=\frac{\phi(p^{s-i'})}{2}\phi(q^{t-m}).\]
\item If $q$ is a quadratic nonresidue modulo $p$, then for $j<m\leqslant t$ and $2\nmid m-j$, we have\[\#(a+C_{i'j}^*)\cap C_{im}^*=\frac{\phi(p^{s-i'})}{2}\phi(q^{t-m}).\]For $j<m\leqslant t$ and $2\mid m-j$, we have\[\#(a+C_{i'j}^*)\cap C_{im}=\frac{\phi(p^{s-i'})}{2}\phi(q^{t-m}).\]
\end{enumerate}
\end{lemma}

\begin{lemma}
Let $a^*\in C_{ij}^*$, $0\leqslant i<i'\leqslant s-1$ and $0\leqslant j\leqslant t-1$, then
\begin{enumerate}[(i)]
\item $\#(a^*+C_{i'j})\cap C_{ij}^*=\frac{\phi(p^{s-i'})}{2}(q-2)q^{t-j-1}$.
\item If $q$ is a quadratic residue modulo $p$, then for $j<m\leqslant t$, we have
\[\#(a^*+C_{i'j}^*)\cap C_{im}^*=\frac{\phi(p^{s-i'})}{2}\phi(q^{t-m}).\]
\item If $q$ is a quadratic nonresidue modulo $p$, then for $j<m\leqslant t$ and $2\nmid m-j$, we have\[\#(a^*+C_{i'j})\cap C_{im}=\frac{\phi(p^{s-i'})}{2}\phi(q^{t-m}).\]For $j<m\leqslant t$ and $2\mid m-j$, we have\[\#(a^*+C_{i'j})\cap C_{im}^*=\frac{\phi(p^{s-i'})}{2}\phi(q^{t-m}).\]
\end{enumerate}\label{3.17}
\end{lemma}

\section{Arithmetic properties}\label{sec4}

In this section, we will study arithmetic properties of some families in the ring $\mathbb{F}_l[x]/\langle x^{p^sq^t-1}\rangle$. For that, let us define \[\chi_{ij}=\sum_{\alpha\in C_{ij}}x^\alpha\]and\[\chi_{ij}^*=\sum_{\alpha\in C_{ij}^*}x^\alpha.\]In particular, $\chi_{sj}=\chi_{sj}^*$ for $0\leqslant j\leqslant t$ and $\chi_{st}=1$.

\begin{theorem}
Let $0\leqslant i\leqslant s-1$ and $0\leqslant j\leqslant t-1$, then
\begin{enumerate}[(i)]
\item If $q$ is a quadratic residue modulo $p$, then
\[
\begin{aligned}\chi_{ij}^2=&\frac{p-3}{4}p^{s-i-1}\left(\left(q-2\right)q^{t-j-1}\chi_{ij}+\phi(q^{t-j})\sum_{j<m\leqslant t}\chi_{im}\right)\\&+\frac{p+1}{4}p^{s-i-1}\left(\left(q-2\right)q^{t-j-1}\chi_{ij}^*+\phi(q^{t-j})\sum_{j<m\leqslant t}\chi_{im}^*\right).
\end{aligned}\]\label{1}
\item If $q$ is a quadratic nonresidue modulo $p$, then
\[
\begin{aligned}\chi_{ij}^2=&\frac{p-3}{4}p^{s-i-1}\left(\left(q-2\right)q^{t-j-1}\chi_{ij}+\phi(q^{t-j})\left(\sum_{\substack{j<m\leqslant t\\2\nmid m-j}}\chi_{im}^*+\sum_{\substack{j<m\leqslant t\\2\mid m-j}}\chi_{im}\right)\right)\\
&+\frac{p+1}{4}p^{s-i-1}\left(\left(q-2\right)q^{t-j-1}\chi_{ij}^*+\phi(q^{t-j})\left(\sum_{\substack{j<m\leqslant t\\2\nmid m-j}}\chi_{im}+\sum_{\substack{j<m\leqslant t\\2\mid m-j}}\chi_{im}^*\right)\right).
\end{aligned}\]
\end{enumerate}\label{4.1}
\end{theorem}

\begin{proof}
We will only prove result (\ref{1}). Let $a\in C_{ij}$, then by Lemma \ref{3.2}, we obtain \begin{equation}
\#(a+C_{ij})\cap C_{im}=\frac{p-3}{4}p^{s-i-1}\phi(q^{t-m}),\quad j<m\leqslant t.\label{eq1}
\end{equation}
We can choose $b\in C_{ij}$ such that\[a+b=c,\]where $c\in C_{im}$. So we get $l^ka+l^kb=l^kc$ for $0\leqslant k\leqslant \frac{\phi(p^{s-i})}{2}\phi(q^{t-j})-1$. When $l^ka$ runs through $C_{ij}$, for $\#C_{im}=\frac{p^{s-i}}{2}\phi(q^{t-m})$, $l^kc$ runs through $C_{im}$ exactly $\# C_{ij}/\# C_{im}=\phi(q^{t-j})/\phi(q^{t-m})$ times. So by \eqref{eq1}, we obtain that the multiset $C_{ij}+C_{ij}$ contains $\frac{p-3}{4}p^{s-i-1}\phi(q^{t-j})$ copies of $C_{im}$. By similar argument and by Lemma \ref{3.2}, we obtain that the multiset $C_{ij}+C_{ij}$ contains $\frac{p-3}{4}p^{s-i-1}(q-2)q^{t-j-1}$ copies of $C_{ij}$, $\frac{p+1}{4}p^{s-i-1}(q-2)q^{t-j-1}$ copies of $C_{ij}^*$ and $\frac{p+1}{4}p^{s-i-1}\phi(q^{t-j})$ copies of $C_{im}^*$ for $j<m\leqslant t$. 

Further it is easy to verify that the coefficients of the terms on the both sides in result (\ref{1}) are equal. This together with argument above proves the result (\ref{1}).
\end{proof}

On the similar lines in proof of Theorem \ref{4.1}, we can prove the results in Theorem \ref{4.2} $\sim$  Theorem \ref{4.6}.

\begin{theorem}
Let $0\leqslant i\leqslant s-1$ and $0\leqslant j\leqslant t-1$, then
\begin{enumerate}[(i)]
\item If $q$ is a quadratic residue modulo $p$, then
\[
\begin{aligned}\left(\chi_{ij}^*\right)^2=&\frac{p-3}{4}p^{s-i-1}\left(\left(q-2\right)q^{t-j-1}\chi_{ij}^*+\phi(q^{t-j})\sum_{j<m\leqslant t}\chi_{im}^*\right)\\&+\frac{p+1}{4}p^{s-i-1}\left(\left(q-2\right)q^{t-j-1}\chi_{ij}+\phi(q^{t-j})\sum_{j<m\leqslant t}\chi_{im}\right).
\end{aligned}\]
\item If $q$ is a quadratic nonresidue modulo $p$, then
\[
\begin{aligned}\left(\chi_{ij}^*\right)^2=&\frac{p-3}{4}p^{s-i-1}\left(\left(q-2\right)q^{t-j-1}\chi_{ij}^*+\phi(q^{t-j})\left(\sum_{\substack{j<m\leqslant t\\2\nmid m-j}}\chi_{im}+\sum_{\substack{j<m\leqslant t\\2\mid m-j}}\chi_{im}^*\right)\right)\\
&+\frac{p+1}{4}p^{s-i-1}\left(\left(q-2\right)q^{t-j-1}\chi_{ij}+\phi(q^{t-j})\left(\sum_{\substack{j<m\leqslant t\\2\nmid m-j}}\chi_{im}^*+\sum_{\substack{j<m\leqslant t\\2\mid m-j}}\chi_{im}\right)\right).
\end{aligned}\]
\end{enumerate}\label{4.2}
\end{theorem}

\begin{theorem}
Let $0\leqslant i\leqslant s-1$ and $0\leqslant j\leqslant t-1$, then
\[\begin{aligned}\chi_{ij}\chi_{ij}^*=&\frac{p-3}{4}p^{s-i-1}(q-2)q^{t-j-1}(\chi_{ij}+\chi_{ij}^*)+\frac{p-3}{4}p^{s-i-1}\phi(q^{t-j})\sum_{j<m\leqslant t}(\chi_{im}+\chi_{im}^*)\\&
+\frac{\phi(p^{s-i})}{2}\phi(q^{t-j})\sum_{\substack{i<k\leqslant s-1\\j<m\leqslant t}}\left(\chi_{km}+\chi_{km}^*\right)+\frac{\phi(p^{s-i})}{2}(q-2)q^{t-j-1}\chi_{sj}\\&
+\frac{\phi(q^{s-i})}{2}\phi(q^{t-j})\sum_{j<m\leqslant t}\chi_{sm}+\frac{\phi(p^{s-i})}{2}(q-2)q^{t-j-1}\sum_{i<k\leqslant s-1}(\chi_{kj}+\chi_{kj}^*)
\end{aligned}\]
\end{theorem}

\begin{theorem}
Let $0\leqslant j \leqslant t-1$, then
$\chi_{sj}^2=(q-2)q^{t-j-1}\chi_{sj}+\phi(q^{t-j})\sum\limits_{j<m\leqslant t}\chi_{sm}$.
\end{theorem}

\begin{theorem}
Let $0\leqslant i\leqslant s-1$, then
\begin{enumerate}[(i)]
\item $\chi_{it}^2=\frac{p-3}{4}p^{s-i-1}\chi_{it}+\frac{p+1}{4}p^{s-i-1}\chi_{it}^*$.
\item $\left(\chi_{it}^*\right)^2=\frac{p-3}{4}p^{s-i-1}\chi_{it}^*+\frac{p+1}{4}p^{s-i-1}\chi_{it}$.
\item $\chi_{it}\chi_{it}^*=\frac{p-3}{4}p^{s-i-1}\left(\chi_{it}+\chi_{it}^*\right)+\frac{\phi(p^{s-i})}{2}\left(\sum\limits_{i<k\leqslant s-1}\left(\chi_{kt}+\chi_{kt}^*\right)+1\right).$
\end{enumerate}
\end{theorem}

\begin{theorem}
Let $0\leqslant i<i'\leqslant s-1$ and $0\leqslant j<j'\leqslant t$, then
\begin{enumerate}[(i)]
\item $\chi_{ij}\chi_{i'j'}=\chi_{ij}\chi_{i'j'}^*=\frac{\phi(p^{s-i'})}{2}\phi(q^{t-j'})\chi_{ij}$.
\item $\chi_{ij}^*\chi_{i'j'}=\chi_{ij}^*\chi_{i'j'}^*=\frac{\phi(p^{s-i'})}{2}\phi(q^{t-j'})\chi_{ij}^*$.
\end{enumerate}\label{4.6}
\end{theorem}

\begin{theorem}
Let $0\leqslant i<i'\leqslant s-1$ and $0\leqslant j'<j\leqslant t$, then
\begin{enumerate}[(i)]
\item $\chi_{ij}\chi_{i'j'}=\chi_{ij}\chi_{i'j'}^*=\frac{\phi(p^{s-i'})}{2}\phi(q^{t-j})\chi_{ij'}$.\label{2}
\item $\chi_{ij}^*\chi_{i'j'}=\chi_{ij}^*\chi_{i'j'}^*=\frac{\phi(p^{s-i'})}{2}\phi(q^{t-j})\chi_{ij'}^*$.
\end{enumerate}\label{4.7}
\end{theorem}

\begin{proof}
We will prove the first part of result (\ref{2}) for the case $q$ is a quadratic residue modulo $p$. That means, we need to show that, in the multiset\[p^iq^{j'}\left(\{q^{t-j'}x+p^{s-i}q^{j-j'}y:x\in R_{s-i},y\in \mathbb{Z}_{q^{t-j}}^*\}+\{p^{i'-i}q^{t-j'}x'+p^{s-i}y':x'\in R_{s-i'},y'\in \mathbb{Z}_{q^{t-j'}}^*\}\right),\]
every element exactly appears $\frac{\phi(p^{s-i'})}{2}\phi(q^{t-j})$ times.

Note that \[C_{ij'}=p^iq^{j'}\{q^{t-j'}x+p^{s-i}y':x\in R_{s-i},y'\in \mathbb{Z}_{q^{t-j'}}^*\},\]Let $x_0\in R_{s-i}$, since $i>i$, for any $x'\in R_{s-i'}$, there is exactly one $x\in R_{s-i}$ such that $x_0\equiv p^{i'-i}x'+x \mathrm{mod}  p^{s-i}$. Also, let $y_0\in \mathbb{Z}_{q^{t-j}}^*$, for any $y\in \mathbb{Z}_{q^{t-j}}^*$, there is exactly one $y'\in\mathbb{Z}_{q^{t-j'}}^*$ such that $y_0\equiv q^{j-j'}y+y' \mathrm{mod}  q^{t-j'}$. Hence every element appears \[\# R_{s-i'}\cdot\#\mathbb{Z}_{q^{t-j}}^*=\frac{\phi(p^{s-i'})}{2}\phi(q^{t-j})\] times. The result is proved.
\end{proof}

The results in Theorem \ref{4.8} $\sim$ Theorem \ref{4.16} can be proved by using the methods in the proof of Theorem \ref{4.1} or Theorem \ref{4.7}.

\begin{theorem}
Let $0\leqslant i\leqslant s-1$ and $0\leqslant j<j'\leqslant t$, then
\begin{enumerate}[(i)]
\item If $q$ is a quadratic residue modulo $p$, then
\[\chi_{ij}\chi_{ij'}=\frac{p-3}{4}p^{s-i-1}\phi(q^{t-j'})\chi_{ij}+\frac{p+1}{4}p^{s-i-1}\phi(q^{t-j'})\chi_{ij}^*.\]
\item If $q$ is a quadratic nonresidue modulo $p$ and $2\mid j'-j$, then
\[\chi_{ij}\chi_{ij'}=\frac{p-3}{4}p^{s-i-1}\phi(q^{t-j'})\chi_{ij}+\frac{p+1}{4}p^{s-i-1}\phi(q^{t-j'})\chi_{ij}^*.\]
\item If $q$ is a quadratic nonresidue modulo $p$ and $2\nmid j'-j$, then
\[\begin{aligned}
\chi_{ij}\chi_{ij'}=&\frac{p-3}{4}p^{s-i-1}\phi(q^{t-j'})\left(\chi_{ij}+\chi_{ij}^*\right)\\&+\frac{\phi(p^{s-i})}{2}\phi(q^{t-j'})\left(\sum_{i< k\leqslant s-1}\left(\chi_{kj}+\chi_{kj}^*\right)+\chi_{sj}\right).
\end{aligned}\]
\end{enumerate}\label{4.8}
\end{theorem}

\begin{theorem}
Let $0\leqslant i\leqslant s-1$ and $0\leqslant j<j'\leqslant t$, then
\begin{enumerate}[(i)]
\item If $q$ is a quadratic residue modulo $p$, then
\[\chi_{ij}^*\chi_{ij'}^*=\frac{p-3}{4}p^{s-i-1}\phi(q^{t-j'})\chi_{ij}^*+\frac{p+1}{4}p^{s-i-1}\phi(q^{t-j'})\chi_{ij}.\]
\item If $q$ is a quadratic nonresidue modulo $p$ and $2\mid j'-j$, then
\[\chi_{ij}^*\chi_{ij'}^*=\frac{p-3}{4}p^{s-i-1}\phi(q^{t-j'})\chi_{ij}^*+\frac{p+1}{4}p^{s-i-1}\phi(q^{t-j'})\chi_{ij}.\]
\item If $q$ is a quadratic nonresidue modulo $p$ and $2\nmid j'-j$, then
\[\begin{aligned}
\chi_{ij}^*\chi_{ij'}^*=&\frac{p-3}{4}p^{s-i-1}\phi(q^{t-j'})\left(\chi_{ij}+\chi_{ij}^*\right)\\&+\frac{\phi(p^{s-i})}{2}\phi(q^{t-j'})\left(\sum_{i< k\leqslant s-1}\left(\chi_{kj}+\chi_{kj}^*\right)+\chi_{sj}\right).
\end{aligned}\]
\end{enumerate}
\end{theorem}

\begin{theorem}
Let $0\leqslant i\leqslant s-1$ and $0\leqslant j<j'\leqslant t$, then
\begin{enumerate}[(i)]
\item If $q$ is a quadratic residue modulo $p$, then
\[\begin{aligned}
\chi_{ij}\chi_{ij'}^*=&\frac{p-3}{4}p^{s-i-1}\phi(q^{t-j'})\left(\chi_{ij}+\chi_{ij}^*\right)\\&+\frac{\phi(p^{s-i})}{2}\phi(q^{t-j'})\left(\sum_{i< k\leqslant s-1}\left(\chi_{kj}+\chi_{kj}^*\right)+\chi_{sj}\right).
\end{aligned}\]
\item If $q$ is a quadratic nonresidue modulo $p$ and $2\nmid j'-j$, then
\[\chi_{ij}\chi_{ij'}^*=\frac{p-3}{4}p^{s-i-1}\phi(q^{t-j'})\chi_{ij}+\frac{p+1}{4}p^{s-i-1}\phi(q^{t-j'})\chi_{ij}^*.\]
\item If $q$ is a quadratic nonresidue modulo $p$ and $2\mid j'-j$, then
\[\begin{aligned}
\chi_{ij}\chi_{ij'}^*=&\frac{p-3}{4}p^{s-i-1}\phi(q^{t-j'})\left(\chi_{ij}+\chi_{ij}^*\right)\\&+\frac{\phi(p^{s-i})}{2}\phi(q^{t-j'})\left(\sum_{i< k\leqslant s-1}\left(\chi_{kj}+\chi_{kj}^*\right)+\chi_{sj}\right).
\end{aligned}\]
\end{enumerate}
\end{theorem}

\begin{theorem}
Let $0\leqslant i\leqslant s-1$ and $0\leqslant j<j'\leqslant t$, then
\begin{enumerate}[(i)]
\item If $q$ is a quadratic residue modulo $p$, then
\[\begin{aligned}
\chi_{ij}^*\chi_{ij'}=&\frac{p-3}{4}p^{s-i-1}\phi(q^{t-j'})\left(\chi_{ij}+\chi_{ij}^*\right)\\&+\frac{\phi(p^{s-i})}{2}\phi(q^{t-j'})\left(\sum_{i< k\leqslant s-1}\left(\chi_{kj}+\chi_{kj}^*\right)+\chi_{sj}\right).
\end{aligned}\]
\item If $q$ is a quadratic nonresidue modulo $p$ and $2\nmid j'-j$, then
\[\chi_{ij}^*\chi_{ij'}=\frac{p-3}{4}p^{s-i-1}\phi(q^{t-j'})\chi_{ij}^*+\frac{p+1}{4}p^{s-i-1}\phi(q^{t-j'})\chi_{ij}.\]
\item If $q$ is a quadratic nonresidue modulo $p$ and $2\mid j'-j$, then
\[\begin{aligned}
\chi_{ij}^*\chi_{ij'}=&\frac{p-3}{4}p^{s-i-1}\phi(q^{t-j'})\left(\chi_{ij}+\chi_{ij}^*\right)\\&+\frac{\phi(p^{s-i})}{2}\phi(q^{t-j'})\left(\sum_{i< k\leqslant s-1}\left(\chi_{kj}+\chi_{kj}^*\right)+\chi_{sj}\right).
\end{aligned}\]
\end{enumerate}
\end{theorem}

\begin{theorem}
\begin{enumerate}[(i)]
\item Let $0\leqslant j<j'\leqslant t$, then $\chi_{sj}\chi_{sj'}=\phi(q^{t-j'})\chi_{sj}$.
\item Let $0\leqslant i\leqslant s-1$ and $0\leqslant j< j'\leqslant t$, then\[\chi_{ij}\chi_{sj'}=\phi(q^{t-j'})\chi_{ij}\]and\[\chi_{ij}^*\chi_{sj'}=\phi(q^{t-j'})\chi_{ij}^*.\]
\item Let $0\leqslant i\leqslant s-1$ and $0\leqslant j'< j\leqslant t$, then\[\chi_{ij}\chi_{sj'}=\phi(q^{t-j})\chi_{ij'}\]and\[\chi_{ij}^*\chi_{sj'}=\phi(q^{t-j})\chi_{ij'}^*.\]
\end{enumerate}
\end{theorem}

\begin{theorem}
Let $0\leqslant i<i'\leqslant s-1$ and $0\leqslant j\leqslant t-1$, then
\begin{enumerate}[(i)]
\item If $q$ is a quadratic residue modulo $p$, then\[\chi_{ij}\chi_{i'j}=\frac{\phi(p^{s-i'})}{2}(q-2)q^{t-j-1}\chi_{ij}+\frac{\phi(p^{s-i'})}{2}\phi(q^{t-j})\sum_{j<m\leqslant t}\chi_{im}.\]
\item If $q$ is a quadratic nonresidue modulo $p$, then
\[
\chi_{ij}\chi_{i'j}=\frac{\phi(p^{s-i'})}{2}(q-2)q^{t-j-1}\chi_{ij}+\frac{\phi(p^{s-i'})}{2}\phi(q^{t-j})\left(\sum_{\substack{j<m\leqslant t\\2\nmid m-j}}\chi_{im}^*+\sum_{\substack{j<m\leqslant t\\2\mid m-j}}\chi_{im}\right).
\]
\end{enumerate}
\end{theorem}

\begin{theorem}
Let $0\leqslant i<i'\leqslant s-1$ and $0\leqslant j\leqslant t-1$, then
\begin{enumerate}[(i)]
\item If $q$ is a quadratic residue modulo $p$, then\[\chi_{ij}^*\chi_{i'j}^*=\frac{\phi(p^{s-i'})}{2}(q-2)q^{t-j-1}\chi_{ij}^*+\frac{\phi(p^{s-i'})}{2}\phi(q^{t-j})\sum_{j<m\leqslant t}\chi_{im}^*.\]
\item If $q$ is a quadratic nonresidue modulo $p$, then
\[
\chi_{ij}^*\chi_{i'j}^*=\frac{\phi(p^{s-i'})}{2}(q-2)q^{t-j-1}\chi_{ij}^*+\frac{\phi(p^{s-i'})}{2}\phi(q^{t-j})\left(\sum_{\substack{j<m\leqslant t\\2\nmid m-j}}\chi_{im}+\sum_{\substack{j<m\leqslant t\\2\mid m-j}}\chi_{im}^*\right).
\]
\end{enumerate}
\end{theorem}

\begin{theorem}
Let $0\leqslant i<i'\leqslant s-1$ and $0\leqslant j\leqslant t-1$, then
\begin{enumerate}[(i)]
\item If $q$ is a quadratic residue modulo $p$, then\[\chi_{ij}\chi_{i'j}^*=\frac{\phi(p^{s-i'})}{2}(q-2)q^{t-j-1}\chi_{ij}+\frac{\phi(p^{s-i'})}{2}\phi(q^{t-j})\sum_{j<m\leqslant t}\chi_{im}.\]
\item If $q$ is a quadratic nonresidue modulo $p$, then
\[
\chi_{ij}\chi_{i'j}^*=\frac{\phi(p^{s-i'})}{2}(q-2)q^{t-j-1}\chi_{ij}+\frac{\phi(p^{s-i'})}{2}\phi(q^{t-j})\left(\sum_{\substack{j<m\leqslant t\\2\nmid m-j}}\chi_{im}^*+\sum_{\substack{j<m\leqslant t\\2\mid m-j}}\chi_{im}\right).
\]
\end{enumerate}
\end{theorem}

\begin{theorem}
Let $0\leqslant i<i'\leqslant s-1$ and $0\leqslant j\leqslant t-1$, then
\begin{enumerate}[(i)]
\item If $q$ is a quadratic residue modulo $p$, then\[\chi_{ij}^*\chi_{i'j}=\frac{\phi(p^{s-i'})}{2}(q-2)q^{t-j-1}\chi_{ij}^*+\frac{\phi(p^{s-i'})}{2}\phi(q^{t-j})\sum_{j<m\leqslant t}\chi_{im}^*.\]
\item If $q$ is a quadratic nonresidue modulo $p$, then
\[
\chi_{ij}^*\chi_{i'j}=\frac{\phi(p^{s-i'})}{2}(q-2)q^{t-j-1}\chi_{ij}^*+\frac{\phi(p^{s-i'})}{2}\phi(q^{t-j})\left(\sum_{\substack{j<m\leqslant t\\2\nmid m-j}}\chi_{im}+\sum_{\substack{j<m\leqslant t\\2\mid m-j}}\chi_{im}^*\right).
\]
\end{enumerate}\label{4.16}
\end{theorem}

\begin{corollary}
Let $0\leqslant i<i'\leqslant s-1$ and $0\leqslant j\leqslant t-1$, then\[\chi_{ij}\chi_{i'j}=\chi_{ij}\chi_{i'j}^*\]and\[\chi_{ij}^*\chi_{i'j}=\chi_{ij}^*\chi_{i'j}^*.\]
\end{corollary}

\section{Primitive idempotents in the ring $\frac{\mathbb{F}_l[x]}{\langle x^{p^sq^t}-1\rangle}$}
\label{sec5}

Let $\alpha$ be a fixed primitive $n$th root of unit in some extension field of $\mathbb{F}_l$, then the polynomial\[M_{\gamma}(x)=\prod_{i\in C_\gamma}(x-\alpha^i)\]is the minimal polynomial of $\alpha^\gamma$ over $\mathbb{F}_l$. Let $\mathscr{M}_\gamma$ be the minimal ideal in $\frac{\mathbb{F}_l[x]}{\langle x^n-1\rangle}$ generated by $\frac{x^n-1}{M_\gamma(x)}$ and $\theta_\gamma(x)$ be the primitive idempotent of $\mathscr{M}_\gamma$, that is,\[
\theta_\gamma(\alpha^u)=\bigg\{\begin{array}{ll}1,&u\in C_\gamma;\\0,&u\notin C_\gamma.\end{array}
\]

\begin{lemma}[\cite{ref4}, Theorem 1, p.433]
\[\theta_\gamma(x)=\sum_{0\leqslant u<n}\epsilon_u x^u\]where\[\epsilon_u=1/n\sum_{j\in C_\gamma}\alpha^{-uj},\quad i\geqslant 0.\]\label{5.1}
\end{lemma}

In this section, let $\alpha$ be a fixed primitive $p^sq^t$th root of unit in some extension field of $\mathbb{F}_l$. Also, for simplicity, we denote\[\theta_{ij}=\theta_{p^iq^j}\]and \[\theta_{ij}^*=\theta_{p^iq^jg}\]In particular, $\theta_{ij}=\theta_{ij}^*$ when $i=s$ and $\theta_{st}=\theta_0$.

\subsubsection*{Case 1: $q$ is a quadratic residue modulo $p$}

We first calculate $\chi_{ij}(\alpha^u),\, 0\leqslant i\leqslant s-1, 0\leqslant j\leqslant t$. By Theorem \ref{2.12}, we get
\[\begin{aligned}
\chi_{ij}(\alpha^u)&=\sum_{x\in R_{s-i}}\sum_{y\in \mathbb{Z}_{q^{t-j}}^*}\alpha^{up^iq^j(q^{t-j}x+p^{s-i}y)}=\sum_{x\in R{s-i}}\alpha^{up^iq^tx}\sum_{y\in\mathbb{Z}_{q^{t-j}}^*}\alpha^{up^sq^jy}\\
&=\sum_{x\in R_1}\alpha^{up^iq^tx}\sum_{0\leqslant\lambda<p^{s-i-1}}\alpha^{up^{i+1}q^t\lambda}\sum_{y\in\mathbb{Z}_{q^{t-j}}^*}\alpha^{up^sq^jy}.
\end{aligned}\]
Similarly, we have
\[\chi^*_{ij}(\alpha^u)=\sum_{x\in N_1}\alpha^{up^iq^tx}\sum_{0\leqslant\lambda<p^{s-i-1}}\alpha^{up^{i+1}q^t\lambda}\sum_{y\in\mathbb{Z}_{q^{t-j}}^*}\alpha^{up^sq^jy}.\]
The results below are obvious:

\begin{lemma}
If $p^{s-i}q^{t-j}\mid u$, then $\chi_{ij}(\alpha^u)=\chi_{ij}^*(\alpha^u)=\frac{\phi(p^{s-i})}{2}\phi(q^{t-j})$.
\end{lemma}

\begin{lemma}
If $p^{s-i-1}\nmid u$ or $q^{t-j-1}\nmid u$, then $\chi_{ij}(\alpha^u)=\chi_{ij}^*(\alpha^u)=0$.
\end{lemma}

\begin{lemma}
If $q^{t-j-1}\| u$, then\[\sum_{y\in\mathbb{Z}_{q^{t-j}}^*}\alpha^{up^sq^jy}=-q^{t-j-1}.\]
\end{lemma}

Let $\beta=\alpha^{p^{s-1}q^t}$, then $\beta$ is a $p$th root of unit. We denote \[\mathscr{R}=\sum_{x\in R_1}\beta^x\]and \[\mathscr{N}=\sum_{x\in N_1}\beta^x,\]then\[\mathscr{R}+\mathscr{N}=-1\quad \text{and} \quad \mathscr{R}\mathscr{N}=\frac{p+1}{4}.\]
So $\mathscr{R},\mathscr{N}$ are two roots of the equation\[x^2+x+\frac{p+1}{4}=0.\]
If $l=2$, then $p\equiv-1 \mathrm{mod}  8$ for $2\in R_1$. Without loss of generality, we may suppose\[\mathscr{R}=1,\quad \mathscr{N}=0.\]
If $l$ is an odd prime, then\[\mathscr{R}=\frac{-1+\delta}{2},\quad \mathscr{N}=\frac{-1-\delta}{2}\]where $\delta\in\mathbb{F}_l$ such that $\delta^2=-p$. Note that, by quadratic reciprocity law, \[\left(\frac{-p}{l}\right)=\left(\frac{l}{p}\right)=1,\]so $\delta$ is well defined.

\begin{lemma}Let $q$ be a quadratic residue modulo $p$.
\begin{enumerate}[(i)]
\item If $u\in C_{s-i-1,t-j-1}$, then \[
\chi_{ij}(\alpha^u)=-\mathscr{R} p^{s-i-1}q^{t-j-1}
\]and\[\chi_{ij}^*(\alpha^u)=-\mathscr{N} p^{s-i-1}q^{t-j-1}.\]
\item If $u\in C_{s-i-1,t-j-1}^*$, then \[
\chi_{ij}(\alpha^u)=-\mathscr{N} p^{s-i-1}q^{t-j-1}
\]and\[\chi_{ij}^*(\alpha^u)=-\mathscr{R} p^{s-i-1}q^{t-j-1}.\]
\end{enumerate}
\end{lemma}

\begin{lemma}Let q be a quadratic residue modulo $p$ and $t-j\leqslant m\leqslant t$.
\begin{enumerate}[(i)]
\item If $u\in C_{s-i-1,m}$, then \[
\chi_{ij}(\alpha^u)=\mathscr{R} p^{s-i-1}\phi(q^{t-j})
\]and\[\chi_{ij}^*(\alpha^u)=\mathscr{N} p^{s-i-1}\phi(q^{t-j}).\]
\item If $u\in C_{s-i-1,m}^*$, then \[
\chi_{ij}(\alpha^u)=\mathscr{N} p^{s-i-1}\phi(q^{t-j})
\]and\[\chi_{ij}^*(\alpha^u)=\mathscr{R} p^{s-i-1}\phi(q^{t-j}).\]
\end{enumerate}
\end{lemma}
If $i=s$, then
\[\chi_{sj}(\alpha^u)=\sum_{y\in \mathbb{Z}_{q^{t-j}}^*}\alpha^{up^sq^jy},\]
these values are easy to calculate.

Now we consider the explicit expressions for primitive idempotents. Note that if $u$ and $u'$ are in same cyclotomic coset, the coefficient of $x^u$ and $x^{u'}$ are also same. so we can evaluate, for example $\theta_{ij}$,  by classifying $u$ according to different cyclotomic cosets, that is,
\begin{equation}\theta_{ij}=\sum_{\gamma\in \Gamma}a_\gamma\chi_\gamma\label{eq}\end{equation}
where $\Gamma$ is the set of  the representatives of all cyclotomic cosets. By Lemma \ref{5.1} we know that\[a_\gamma=\frac{1}{n}\chi_{ij}(\alpha^{-u}),\,-u\in C_\gamma\]then by the argument above, we can get the results below.

\begin{theorem}
Let q be a quadratic residue modulo $p$.
\begin{enumerate}[(i)]
\item For $i=0$ and $j=t$,
\[\theta_{0t}(x)=\frac{1}{pq^t}\left(\frac{p-1}{2}\sum_{0\leqslant m\leqslant t}\chi_{sm}(x)+\mathscr{N}\sum_{0\leqslant m\leqslant t}\chi_{s-1,m}(x)+\mathscr{R}\sum_{0\leqslant m\leqslant t}\chi_{s-1,m}^*(x)\right),\]
and
\[\theta_{0t}^*(x)=\frac{1}{pq^t}\left(\frac{p-1}{2}\sum_{0\leqslant m\leqslant t}\chi_{sm}(x)+\mathscr{R}\sum_{0\leqslant m\leqslant t}\chi_{s-1,m}(x)+\mathscr{N}\sum_{0\leqslant m\leqslant t}\chi_{s-1,m}^*(x)\right),\]

\item For $i=0$ and $0\leqslant j\leqslant t-1$,
\[\begin{aligned}
\theta_{0j}(x)=&\frac{1}{pq^{j+1}}\Bigg(\frac{(p-1)}{2}(q-1)\sum_{t-j\leqslant m\leqslant t}\chi_{sm}(x)\\
&+(q-1)\bigg(\mathscr{N}\sum_{t-j\leqslant m\leqslant t}\chi_{s-1,m}(x)+\mathscr{R}\sum_{t-j\leqslant m\leqslant m}\chi_{s-1,m}^*(x)\bigg)\\
&-\bigg(\mathscr{N}\chi_{s-1,t-j-1}(x)+\mathscr{R}\chi_{s-1,t-j-1}^*(x)\bigg)\\
&-\frac{p-1}{2}\chi_{s,t-j-1}(x)\Bigg),
\end{aligned}\]
and
\[\begin{aligned}
\theta_{0j}^*(x)=&\frac{1}{pq^{j+1}}\Bigg(\frac{(p-1)}{2}(q-1)\sum_{t-j\leqslant m\leqslant t}\chi_{sm}(x)\\
&+(q-1)\bigg(\mathscr{R}\sum_{t-j\leqslant m\leqslant t}\chi_{s-1,m}(x)+\mathscr{N}\sum_{t-j\leqslant m\leqslant m}\chi_{s-1,m}^*(x)\bigg)\\
&-\bigg(\mathscr{R}\chi_{s-1,t-j-1}(x)+\mathscr{N}\chi_{s-1,t-j-1}^*(x)\bigg)\\
&-\frac{p-1}{2}\chi_{s,t-j-1}(x)\Bigg).
\end{aligned}\]

\item For $1\leqslant i\leqslant s-1 $ and $0\leqslant j\leqslant t-1$
\[\begin{aligned}
\theta_{ij}(x)=&\frac{1}{p^{i+1}q^{j+1}}\Biggl(\frac{(p-1)}{2}(q-1)\bigg(\sum_{\substack{s-i\leqslant k\leqslant s-1\\t-j\leqslant m\leqslant t}}(\chi_{km}(x)+\chi_{km}^*(x))+\sum_{t-j\leqslant m\leqslant t}\chi_{sm}(x)\bigg)\\
&+(q-1)\bigg(\mathscr{N}\sum_{t-j\leqslant m\leqslant t}\chi_{s-i-1,m}(x)+\mathscr{R}\sum_{t-j\leqslant m\leqslant m}\chi_{s-i-1,m}^*(x)\bigg)\\
&-\bigg(\mathscr{N}\chi_{s-i-1,t-j-1}(x)+\mathscr{R}\chi_{s-i-1,t-j-1}^*(x)\bigg)\\
&-\frac{p-1}{2}\bigg(\chi_{s-i,t-j-1}(x)+\chi_{s-i,t-j-1}^*(x)\bigg)\Biggr).
\end{aligned}\]
and
\[\begin{aligned}
\theta_{ij}^*(x)=&\frac{1}{p^{i+1}q^{j+1}}\Bigg(\frac{(p-1)}{2}(q-1)\bigg(\sum_{\substack{s-i\leqslant k\leqslant s-1\\t-j\leqslant m\leqslant t}}(\chi_{km}(x)+\chi_{km}^*(x))+\sum_{t-j\leqslant m\leqslant t}\chi_{sm}(x)\bigg)\\
&+(q-1)\bigg(\mathscr{R}\sum_{t-j\leqslant m\leqslant t}\chi_{s-i-1,m}(x)+\mathscr{N}\sum_{t-j\leqslant m\leqslant m}\chi_{s-i-1,m}^*(x)\bigg)\\
&-\bigg(\mathscr{R}\chi_{s-i-1,t-j-1}(x)+\mathscr{N}\chi_{s-i-1,t-j-1}^*(x)\bigg)\\
&-\frac{p-1}{2}\bigg(\chi_{s-i,t-j-1}(x)+\chi_{s-i,t-j-1}^*(x)\bigg)\Bigg).
\end{aligned}\]

\item For $1\leqslant i\leqslant s-1$ and $j=t$, 
\[\begin{aligned}
\theta_{it}(x)=&
\frac{1}{p^{i+1}q^t}\Bigg(\frac{p-1}{2}\sum_{\substack{s-i\leqslant k\leqslant s-1\\ 0\leqslant m\leqslant t}}(\chi_{km}(x)+\chi_{km}^*(x))+\frac{p-1}{2}\sum_{0\leqslant m\leqslant t}\chi_{sm}(x)\\&+\mathscr{N}\sum_{0\leqslant m\leqslant t}\chi_{s-i-1,m}(x)+\mathscr{R}\sum_{0\leqslant m\leqslant t}\chi_{s-i-1,m}^*(x)\Bigg),
\end{aligned}\]
and
\[\begin{aligned}
\theta_{it}^*(x)=&
\frac{1}{p^{i+1}q^t}\Bigg(\frac{p-1}{2}\sum_{\substack{s-i\leqslant k\leqslant s-1\\ 0\leqslant m\leqslant t}}(\chi_{km}(x)+\chi_{km}^*(x))+\frac{p-1}{2}\sum_{0\leqslant m\leqslant t}\chi_{sm}(x)\\&+\mathscr{R}\sum_{0\leqslant m\leqslant t}\chi_{s-i-1,m}(x)+\mathscr{N}\sum_{0\leqslant m\leqslant t}\chi_{s-i-1,m}^*(x)\Bigg),
\end{aligned}\]

\item For $i=s$ and $0\leqslant j\leqslant t-1$,
\[\begin{aligned}
\theta_{sj}(x)=&\frac{1}{p^sq^{j+1}}\Bigg((q-1)\bigg(\sum_{\substack{0\leqslant k\leqslant s-1\\t-j\leqslant m\leqslant t}}(\chi_{km}(x)+\chi_{km}^*(x))+\sum_{t-j\leqslant m\leqslant t}\chi_{sm}(x)\bigg)\\
&-\bigg(\sum_{0\leqslant k\leqslant s-1}(\chi_{k,t-j-1}(x)+\chi_{k,t-j-1}^*(x))+\chi_{s,t-j-1}(x)\bigg)\Bigg),
\end{aligned}\]

\item For $i=s$ and $j=t$,
\[\begin{aligned}
\theta_{st}(x)&=\frac{1}{p^sq^t}\Bigg(\sum_{\substack{0\leqslant k\leqslant s-1\\0\leqslant m\leqslant t}}(\chi_{ij}(x)+\chi_{ij}^*(x))+\sum_{0\leqslant m\leqslant t}\chi_{sj}(x)\Bigg)\\
&=\frac{1}{p^sq^t}\sum_{0\leqslant u< p^sq^t}x^u.
\end{aligned}\]
\end{enumerate}\label{5.7}
\end{theorem}

\subsubsection*{Case 2: $q$ is a quadratic nonresidue modulo $p$}

Let $q$ be a quadratic nonresidue modulo $p$ and $u\in C_{km}$ or $C_{km}^*$. If $k<i-1$ or $k\geqslant i$, then we can evaluate $\chi_{ij}(\alpha^{-u})$ on the similar lines in Case1 1, for whether $u$ is a quadratic residue modulo $p$ does not influence the coefficient of $\chi_{km}$ or $\chi_{km}^*$. However, if $k=i-1$, by theorem \ref{2.12}, whether the coefficient of $\chi_{km}$ (or $\chi_{km}^*$) is  a multiple of $\mathscr{R}$ or $\mathscr{N}$ is depend on $j-m \mathrm{mod}  2$. Then by \eqref{eq}, we can get the results below.

\begin{theorem}
Let $q$ be a quadratic nonresidue modulo $p$.
\begin{enumerate}[(i)]
\item For $i=0$ and $j=t$,
\[\begin{aligned}\theta_{0t}(x)=&\frac{1}{pq^t}\Bigg(\frac{p-1}{2}\sum_{0\leqslant m\leqslant t}\chi_{sm}(x)\\&+\mathscr{N}\bigg(\sum_{\substack{0\leqslant m\leqslant t\\2\mid t-m}}\chi_{s-1,m}(x)+\sum_{\substack{0\leqslant m\leqslant t\\2\nmid t-m}}\chi_{s-1,m}^*(x)\bigg)\\&+\mathscr{R}\bigg(\sum_{\substack{0\leqslant m\leqslant t\\2\mid t-m}}\chi_{s-1,m}^*(x)+\sum_{\substack{0\leqslant m\leqslant t\\2\nmid t-m}}\chi_{s-1,m}(x)\bigg)\Bigg),\end{aligned}\]
and
\[\begin{aligned}\theta_{0t}^*(x)=&\frac{1}{pq^t}\Bigg(\frac{p-1}{2}\sum_{0\leqslant m\leqslant t}\chi_{sm}(x)\\&+\mathscr{R}\bigg(\sum_{\substack{0\leqslant m\leqslant t\\2\mid t-m}}\chi_{s-1,m}(x)+\sum_{\substack{0\leqslant m\leqslant t\\2\nmid t-m}}\chi_{s-1,m}^*(x)\bigg)\\&+\mathscr{N}\bigg(\sum_{\substack{0\leqslant m\leqslant t\\2\mid t-m}}\chi_{s-1,m}^*(x)+\sum_{\substack{0\leqslant m\leqslant t\\2\nmid t-m}}\chi_{s-1,m}(x)\bigg)\Bigg),\end{aligned}\]

\item For $i=0$ and $0\leqslant j\leqslant t-1$,
\[\begin{aligned}
\theta_{0j}(x)=&\frac{1}{pq^{j+1}}\Bigg(\frac{(p-1)}{2}(q-1)\sum_{t-j\leqslant m\leqslant t}\chi_{sm}(x)\\
&+(q-1)\mathscr{N}\bigg(\sum_{\substack{t-j\leqslant m\leqslant t\\2\mid m-j}}\chi_{s-1,m}(x)+\sum_{\substack{t-j\leqslant m\leqslant t\\2\nmid m-j}}\chi_{s-1,m}^*(x)\bigg)\\
&+(q-1)\mathscr{R}\bigg(\sum_{\substack{t-j\leqslant m\leqslant m\\2\mid m-j}}\chi_{s-1,m}^*(x)+\sum_{\substack{t-j\leqslant m\leqslant m\\2\nmid m-j}}\chi_{s-1,m}(x)\bigg)\\
&-\bigg(\frac{-1+(-1)^{t}\delta}{2}\chi_{s-1,t-j-1}(x)+\frac{-1-(-1)^{t}\delta}{2}\chi_{s-1,t-j-1}^*(x)\bigg)\\
&-\frac{p-1}{2}\chi_{s,t-j-1}(x)\Bigg),
\end{aligned}\]
and
\[\begin{aligned}
\theta_{0j}(x)^*=&\frac{1}{pq^{j+1}}\Bigg(\frac{(p-1)}{2}(q-1)\sum_{t-j\leqslant m\leqslant t}\chi_{sm}(x)\\
&+(q-1)\mathscr{R}\bigg(\sum_{\substack{t-j\leqslant m\leqslant t\\2\mid m-j}}\chi_{s-1,m}(x)+\sum_{\substack{t-j\leqslant m\leqslant t\\2\nmid m-j}}\chi_{s-1,m}^*(x)\bigg)\\
&+(q-1)\mathscr{N}\bigg(\sum_{\substack{t-j\leqslant m\leqslant m\\2\mid m-j}}\chi_{s-1,m}^*(x)+\sum_{\substack{t-j\leqslant m\leqslant m\\2\nmid m-j}}\chi_{s-1,m}(x)\bigg)\\
&-\bigg(\frac{-1-(-1)^{t}\delta}{2}\chi_{s-1,t-j-1}(x)+\frac{-1+(-1)^{t}\delta}{2}\chi_{s-1,t-j-1}^*(x)\bigg)\\
&-\frac{p-1}{2}\chi_{s,t-j-1}(x)\Bigg).
\end{aligned}\]

\item For $1\leqslant i\leqslant s-1 $ and $0\leqslant j\leqslant t-1$
\[\begin{aligned}
\theta_{ij}(x)=&\frac{1}{p^{i+1}q^{j+1}}\Biggl(\frac{(p-1)}{2}(q-1)\bigg(\sum_{\substack{s-i\leqslant k\leqslant s-1\\t-j\leqslant m\leqslant t}}(\chi_{km}(x)+\chi_{km}^*(x))+\sum_{t-j\leqslant m\leqslant t}\chi_{sm}(x)\bigg)\\
&+(q-1)\mathscr{N}\bigg(\sum_{\substack{t-j\leqslant m\leqslant t\\2\mid m-j}}\chi_{s-i-1,m}(x)+\sum_{\substack{t-j\leqslant m\leqslant t\\2\nmid m-j}}\chi_{s-1,m}^*(x)\bigg)\\
&+(q-1)\mathscr{R}\bigg(\sum_{\substack{t-j\leqslant m\leqslant m\\2\mid m-j}}\chi_{s-i-1,m}^*(x)+\sum_{\substack{t-j\leqslant m\leqslant m\\2\nmid m-j}}\chi_{s-1,m}(x)\bigg)\\
&-\bigg(\frac{-1+(-1)^{t}\delta}{2}\chi_{s-i-1,t-j-1}(x)+\frac{-1-(-1)^{t}\delta}{2}\chi_{s-i-1,t-j-1}^*(x)\bigg)\\
&-\frac{p-1}{2}\bigg(\chi_{s-i,t-j-1}(x)+\chi_{s-i,t-j-1}^*(x)\bigg)\Biggr).
\end{aligned}\]
and
\[\begin{aligned}
\theta_{ij}^*(x)=&\frac{1}{p^{i+1}q^{j+1}}\Biggl(\frac{(p-1)}{2}(q-1)\bigg(\sum_{\substack{s-i\leqslant k\leqslant s-1\\t-j\leqslant m\leqslant t}}(\chi_{km}(x)+\chi_{km}^*(x))+\sum_{t-j\leqslant m\leqslant t}\chi_{sm}(x)\bigg)\\
&+(q-1)\mathscr{R}\bigg(\sum_{\substack{t-j\leqslant m\leqslant t\\2\mid m-j}}\chi_{s-i-1,m}(x)+\sum_{\substack{t-j\leqslant m\leqslant t\\2\nmid m-j}}\chi_{s-1,m}^*(x)\bigg)\\
&+(q-1)\mathscr{N}\bigg(\sum_{\substack{t-j\leqslant m\leqslant m\\2\mid m-j}}\chi_{s-i-1,m}^*(x)+\sum_{\substack{t-j\leqslant m\leqslant m\\2\nmid m-j}}\chi_{s-1,m}(x)\bigg)\\
&-\bigg(\frac{-1-(-1)^{t}\delta}{2}\chi_{s-i-1,t-j-1}(x)+\frac{-1+(-1)^{t}\delta}{2}\chi_{s-i-1,t-j-1}^*(x)\bigg)\\
&-\frac{p-1}{2}\bigg(\chi_{s-i,t-j-1}(x)+\chi_{s-i,t-j-1}^*(x)\bigg)\Biggr).
\end{aligned}\]

\item For $1\leqslant i\leqslant s-1$ and $j=t$, 
\[\begin{aligned}
\theta_{it}(x)=&
\frac{1}{p^{i+1}q^t}\Bigg(\frac{p-1}{2}\sum_{\substack{s-i\leqslant k\leqslant s-1\\ 0\leqslant m\leqslant t}}(\chi_{km}(x)+\chi_{km}^*(x))+\frac{p-1}{2}\sum_{0\leqslant m\leqslant t}\chi_{sm}(x)\\&+\mathscr{N}\bigg(\sum_{\substack{0\leqslant m\leqslant t\\2\mid t-m}}\chi_{s-i-1,m}(x)+\sum_{\substack{0\leqslant m\leqslant t\\2\nmid t-m}}\chi_{s-i-1,m}^*(x)\bigg)\\&+\mathscr{R}\bigg(\sum_{\substack{0\leqslant m\leqslant t\\2\mid t-m}}\chi_{s-i-1,m}^*(x)+\sum_{\substack{0\leqslant m\leqslant t\\2\nmid t-m}}\chi_{s-i-1,m}(x)\bigg)\Bigg),
\end{aligned}\]
and
\[\begin{aligned}
\theta_{it}(x)=&
\frac{1}{p^{i+1}q^t}\Bigg(\frac{p-1}{2}\sum_{\substack{s-i\leqslant k\leqslant s-1\\ 0\leqslant m\leqslant t}}(\chi_{km}(x)+\chi_{km}^*(x))+\frac{p-1}{2}\sum_{0\leqslant m\leqslant t}\chi_{sm}(x)\\&+\mathscr{R}\bigg(\sum_{\substack{0\leqslant m\leqslant t\\2\mid t-m}}\chi_{s-i-1,m}(x)+\sum_{\substack{0\leqslant m\leqslant t\\2\nmid t-m}}\chi_{s-i-1,m}^*(x)\bigg)\\&+\mathscr{N}\bigg(\sum_{\substack{0\leqslant m\leqslant t\\2\mid t-m}}\chi_{s-i-1,m}^*(x)+\sum_{\substack{0\leqslant m\leqslant t\\2\nmid t-m}}\chi_{s-i-1,m}(x)\bigg)\Bigg),
\end{aligned}\]

\item For $i=s$ and $0\leqslant j\leqslant t-1$,
\[\begin{aligned}
\theta_{sj}(x)=&\frac{1}{p^sq^{j+1}}\Bigg((q-1)\bigg(\sum_{\substack{0\leqslant k\leqslant s-1\\t-j\leqslant m\leqslant t}}(\chi_{km}(x)+\chi_{km}^*(x))+\sum_{t-j\leqslant m\leqslant t}\chi_{sm}(x)\bigg)\\
&-\bigg(\sum_{0\leqslant k\leqslant s-1}(\chi_{k,t-j-1}(x)+\chi_{k,t-j-1}^*(x))+\chi_{s,t-j-1}(x)\bigg)\Bigg),
\end{aligned}\]

\item For $i=s$ and $j=t$,
\[\begin{aligned}
\theta_{st}(x)&=\frac{1}{p^sq^t}\Bigg(\sum_{\substack{0\leqslant k\leqslant s-1\\0\leqslant m\leqslant t}}(\chi_{ij}(x)+\chi_{ij}^*(x))+\sum_{0\leqslant m\leqslant t}\chi_{sj}(x)\Bigg)\\
&=\frac{1}{p^sq^t}\sum_{0\leqslant u< p^sq^t}x^u.
\end{aligned}\]
\end{enumerate}\label{5.8}
\end{theorem}

\begin{remark}
In Theorem \ref{5.8}, if $l=2$, we set\[\frac{-1+\delta}{2}=\mathscr{R}=1\quad\text{and}\quad\frac{-1-\delta}{2}=\mathscr{N}=0,\]which is respect to define of $\mathscr{R}$ and $\mathscr{N}$ above.
\end{remark}

\begin{example}
Set $p=11$, $q=5$, $s=t=1$, $g=2$ and $l=3$. Then $n=55$. 3 is a primitive of 5 and $\mathrm{ord}_{11}(3)=5$. $R_1=\{1,3,4,5,9,\}$ and $N_1=\{2,6,7,8,10\}$. Note that $5$ is a quadratic residue modulo 11, by Theorem \ref{2.12}, the 3-ary cyclotomic coset modulo are given below:
\[\begin{aligned}
&C_0=\{0\},\\
&C_1=\{5x+11y:x\in R_1,y\in\mathbb{Z}_5^*\}\\&\quad\,=\{1,3,4,9,12,14,16,23,26,27,31,34,36,37,38,42,47,48,49,53\},\\
&C_2=\{5x+11y:x\in N_1,y\in\mathbb{Z}_5^*\}\\&\quad\,=\{2,6,7,8,13,17,18,19,21,24,28,29,32,39,41,43,46,51,52,54\},\\
&C_5=\{5x:x\in R_1\}=\{5,15,20,25,45\}\\
&C_{10}=\{5x:x\in N_1\}=\{10,30,35,40,50\},\\
&C_{11}=\{11y:y\in\mathbb{Z}_5^*\}=\{11,22,33,44\}.
\end{aligned}\]
Set $\delta=2\in\mathbb{F}_3$. Let $\alpha$ be a 55th root of unit such that $\mathscr{R}=2$ and $\mathscr{N}=0$. Note that $11^{-1}=5^{-1}=2$ in $\mathbb{F}_3$, by theorem \ref{5.7}, the six primitives idempotents are given below.
\[\begin{aligned}
&\theta_0(x)=1+x+\cdots+x^{54},\\
&\theta_1(x)=2+2\chi_{10}(x)+\chi_2(x)+\chi_{11}(x),\\
&\theta_2(x)=2+2\chi_{5}(x)+\chi_1(x)+\chi_{11}(x),\\
&\theta_5(x)=2+2\chi_2(x)+2\chi_{10}(x)+2\chi_{11}(x),\\
&\theta_{10}(x)=2+2\chi_1(x)+2\chi_{5}(x)+2\chi_{11}(x),\\
&\theta_11(x)=1+2\chi_1(x)+2\chi_2(x)+\chi_5(x)+\chi_10(x)+2\chi_{11}(x).
\end{aligned}\]
\end{example}

\begin{example}
Set $p=7$, $q=5$, $s=t=1$, $g=3$ and $l=2$. Then $n=35$. 2 is a primitive root of 5 and $\mathrm{ord}_7(2)=3$. $R_1=\{1,2,4\}$ and $N_1=\{3,5,6\}$. Note that  5 is a quadratic nonresidue modulo 7, by Theorem \ref{2.12}, the 2-ary cyclotomic cosets modulo 35 are given below:
\[\begin{aligned}
&C_0=\{0\},\\
&C_1=\{5x+7y:x\in N_1,y\in\mathbb{Z}_5^*\}=\{1,2,4,8,9,11,16,18,22,23,29,32\},\\
&C_3=\{5x+7y:x\in R_1,y\in\mathbb{Z}_5^*\}=\{3,6,12,13,17,19,24,26,27,31,33\},\\
&C_5=\{5x:x\in R_1\}=\{5,10,20\}\\
&C_{15}=\{5x:x\in N_1\}=\{15,25,30\},\\
&C_7=\{7y:y\in\mathbb{Z}_5^*\}=\{7,14,21,28\}.
\end{aligned}\]
Let $\alpha$ be a 35th root of unit such that $\mathscr{R}=1$ and $\mathscr{N}=0$. Note that $7^{-1}=5^{-1}=1$ in $\mathbb{F}_2$, by Theorem  \ref{5.8}, the six primitives idempotents are given below.
\[\begin{aligned}
&\theta_0(x)=1+x+\cdots+x^{34},\\
&\theta_1(x)=\chi_3(x)+\chi_7(x),\\
&\theta_3(x)=\chi_1(x)+\chi_7(x),\\
&\theta_5(x)=1+\chi_3(x)+\chi_5(x)+\chi_7(x),\\
&\theta_{15}(x)=1+\chi_1(x)+\chi_7(x)+\chi_{15}(x),\\
&\theta_7(x)=\chi_1(x)+\chi_3(x)+\chi_7(x).
\end{aligned}\]
\end{example}

\section{Parameters of some cyclic codes of length $p^sq^t$}

The dimension of the minimal code $\mathscr{M}_\gamma$ is the number of non-zeros of the generating idempotent, which is the cardinality of the class $C_\gamma$.

\begin{lemma}[\cite{ref4}, Lemma 10, p. 446]
If $\mathscr{C}$ is a cyclic code of length $n$ generated by $g(x)$ and is of minimum distance $d$, then the code $\hat{\mathscr{C}}$ of length $nk$ generated by $g(x)(1+x^n+\cdots+x^{(k-1)m})$ is a repetition code of $\mathscr{C}$ repeated $k$ times and its minimum distance is $dk$.\label{6.1}
\end{lemma}

\begin{prop}
The generating polynomials of the code $\mathscr{M}_0$ is clearly\[\frac{x^{p^sq^t}-1}{x-1}=1+x+\cdots+x^{p^sq^t-1}\]and its minimum distance is $p^sq^t$.
\end{prop}

\begin{prop}
For $0\leqslant j\leqslant t-1$,
\[x^{p^sq^t}-1=(1+x^{q^{t-j-1}}+\cdots +x^{(q-1)q^{t-j-1}})(x^{q^{t-j-1}}-1)(1+x^{q^{t-j}}+\cdots +x^{(p^sq^j-1)q^{t-j}}).\]
The minimal polynomial of $\alpha^{p^sq^j}$ is\[M_{p^sq^j}(x)=1+x^{q^{t-j-1}}+\cdots +x^{(q-1)q^{t-j-1}},\]
therefore the generating polynomial of $\mathscr{M}_{p^sq^j}$ is
\[(x^{q^{t-j-1}}-1)(1+x^{q^{t-j}}+\cdots +x^{(p^sq^j-1)q^{t-j}}).\]
Let $\mathscr{C}_j$ is the code of length $q^{t-j}$ generated by $x^{q^{t-j-1}}-1$, then by Lemma \ref{6.1}, the code $\mathscr{M}_{p^sq^j}$ is the repetition code of $\mathscr{C}_j$ and its minimum distance is $2p^sq^j$.
\end{prop}

\begin{prop}
For $0\leqslant i\leqslant s-1$ and $0\leqslant j\leqslant t$, we denote $d_{ij}^{(0)}=M_{p^iq^j}(x)$ and $d_{ij}^{(1)}(x)=M_{p^iq^jg}(x)$. Let $A=(a_{ij})\in \mathbb{F}_2^{s\times(t+1)}$. Let $\mathscr{C}_A$ be the code of length $p^sq^t$ generated by
\[g_A(x)=\prod_{\substack{0\leqslant i\leqslant s-1\\0\leqslant j\leqslant t}}d_{ij}^{(a_{ij})(x)},\]then $\mathscr{C}_A$ is a $[p^sq^t,\frac{(p^s+1)q^t}{2}]$ code with minimum odd-like weight $d\geqslant p^{s/2}$. 
In fact, let $A'=(1-a_{ij})_{s\times (t+1)}$ and $\mathscr{C}_{A'}$ be the code generated by\[g_{A'}(x)=\prod_{\substack{0\leqslant i\leqslant s-1\\0\leqslant j\leqslant t}}d_{ij}^{(1-a_{ij})(x)}.\]
Suppose $a(x)$ be a codeword of minimum odd-like weight $d$ of $\mathscr{C}_A$, then the polynomial \[b(x)=a(x)a(x^g) \mathrm{mod} (x^{p^sq^t}-1)\]is a codeword with odd-like weight of both $\mathscr{C}_A$ and $\mathscr{C}_{A'}$. It follows that $b(x)$ is a multiple of
\[g_A(x)g_{A'}(x)=\frac{x^{p^sq^t}-1}{\prod\limits_{0\leqslant j\leqslant t}M_{p^sq^j}(x)}=\frac{x^{p^sq^t-1}}{x^{q^t}-1}=1+x^{q^t}+\cdots+x^{(p^s-1)q^t}.\]Since there are at most $d^2$ terms in $b(x)$, we have $d^2\geqslant p^s$.
\end{prop}

\begin{prop}
Suppose $0\leqslant i\leqslant s-1$ and $0\leqslant j\leqslant t$, let $A=(a_{km})\in \mathbb{F}_2^{(s-1)\times(t+1-j)}$. We define
\[g_A(x)=\prod_{\substack{i\leqslant k\leqslant s-1\\j\leqslant m\leqslant t}}d_{km}^{(a_{km})(x)},\]and let $\mathscr{C}_A$ be the code generated by\[g_A(x)(1+x^{p^{s-i}q^{t-j}}+\cdots+x^{(p^iq^j-1)p^{s-i}q^{t-j}}),\]then $\mathscr{C}_A$ is a $[p^sq^t,\frac{(p^{s-i}+1)q^{t-j}}{2}]$ code with minimum odd-like weight $d\geqslant p^{(s+i)/2}q^j$. 
In fact, let $\bar{\mathscr{C}}_A$ be the code of length $p^{s-i}q^{t-j}$ with the generator polynomial $g_A(x)$. We also define $A'=(1-a_{ij})_{(s-i)\times (t+1-j)}$ and $\bar{\mathscr{C}}_{A'}$ be the code of length $p^{s-i}q^{t-j}$ with the generator polynomial\[g_{A'}(x)=\prod_{\substack{i\leqslant k\leqslant s-1\\j\leqslant m\leqslant t}}d_{km}^{(1-a_{km})(x)}.\]
Suppose $a(x)$ be a codeword of minimum odd-like weight $d_0$ of $\bar{\mathscr{C}}_A$, then the polynomial \[b(x)=a(x)a(x^g) \mathrm{mod} (x^{p^{s-i}q^{t-j}}-1)\]is a codeword with odd-like weight of both $\bar{\mathscr{C}}_A$ and $\bar{\mathscr{C}}_{A'}$. It follows that $b(x)$ is a multiple of
\[g_A(x)g_{A'}(x)=\frac{x^{p^{s-i}q^{t-j}}-1}{\prod\limits_{j\leqslant m\leqslant t}M_{p^sq^m}(x)}=\frac{x^{p^{s-i}q^{t-j}}-1}{x^{q^{t-j}}-1}=1+x^{q^{t-j}}+\cdots+x^{(p^{s-i}-1)q^{t-j}}.\]Since there are at most $d_0^2$ terms in $b(x)$, we have $d_0^2\geqslant p^{s-i}$. Then by Lemma \ref{6.1}, we have\[d=p^iq^jd_0\geqslant p^{(s+i)/2}q^j.\]
\end{prop}

\end{document}